\documentclass[12pt]{article}

\usepackage{amsmath,amsfonts,latexsym,graphicx,
amssymb,amsthm}

\usepackage{color}
\usepackage{youngtab}
\usepackage{authblk}
\usepackage[new]{old-arrows}

\newcommand{\Z}{\mathbb{Z}}

\newcommand{\ad}{\mathrm{ad}}
\newcommand{\Card}{\mathrm{Card }}
\newcommand{\Coker}{\mathrm{Coker}}
\renewcommand{\d}{\mathrm{d}}
\newcommand{\Der}{\mathrm{Der}}

\newcommand{\id}{\mathrm{id}}
\newcommand{\Id}{\mathrm{Id}}

\newcommand{\Ind}{\mathrm{Ind}}

\newcommand{\Inner}{\mathrm{Inner}}
\newcommand{\Ker}{\mathrm{Ker}}

\newcommand{\Res}{\mathrm{Res}}

\newcommand{\Tr}{\mathrm{Tr}}

\newtheorem{thm}{Theorem}
\newtheorem{prop}{Proposition}%[section]
\newtheorem{lemma}{Lemma}%[section]
\newtheorem{cor}{Corollary}%[section]

\newtheorem{conj}{Conjecture}
\newtheorem*{GLtheorem}{Garland-Lepowski Theorem}
\newtheorem*{Shtheorem}{Shirshov's Theorem}
\newtheorem*{SCtheorem}{Shirshov-Cohn Theorem}
\newtheorem*{Mactheorem}{Macdonnald's Theorem}
\newtheorem*{Glennietheorem}{Glennie's Theorem}
\newtheorem*{GlennieIdentity}{Glennie's Identity Theorem}
\newtheorem*{Ctheorem}{Cohn's Reversible Theorem}
\newtheorem*{Kutheorem}{Schreier-Kurosh-Cohn Theorem}
\newtheorem*{conj1weakest}{Conjecture 1 (weakest version)}
\newtheorem*{conj1weak}{Conjecture 1 (weak version)}

\newtheorem*{conj3}{Conjecture 3}
\newtheorem*{thm1imprecise}{Theorem 1 (imprecise version)}

\newtheorem*{thm2}{Theorem 2}

\newcommand{\fd}{\mathfrak{d}}
\newcommand{\fg}{\mathfrak{g}}

\newcommand{\fm}{\mathfrak{m}}

\newcommand{\fS}{\mathfrak{S}}

\newcommand{\fsl}{\mathfrak{sl}}
\newcommand{\fgl}{\mathfrak{gl}}

\newcommand{\Hom}{{\mathrm {Hom}}}

\font\small=cmr10

\title{ On the Free Jordan Algebras}

\author[$\dagger$]{ Iryna Kashuba}

\author [$\ddag$]{Olivier Mathieu}

\affil[$\dagger$]{\small  IME, University of Sao Paulo\\

Rua do Mat\~ao, 1010, 05586080 S\~ao Paulo\\

kashuba@ime.usp.br}

\affil[$\ddag$]{\small Institut Camille Jordan du CNRS\\

Universit\'e de Lyon\\ 

F-69622 Villeurbanne Cedex\\

mathieu@math.univ-lyon1.fr}

\begin{document}

\maketitle

\begin{abstract}
A conjecture for the dimension and the character of
the homogenous components of the free Jordan algebras 
is proposed. As a support of the conjecture, some numerical evidences
are generated by a computer  and some  new theoretical results are  proved. One 
of them is the cyclicity of the Jordan operad.
\end{abstract}

\section*{}{\it Introduction.}
Let $K$ be a field of characteristic zero and let
$J(D)$ be  the free Jordan algebra
with $D$ generators $x_1,\dots x_D$ over $K$. Then

\centerline{
$J(D)=\oplus_{n\geq 1}\, J_n(D)$}

\noindent where $J_n(D)$ consists of all degree $n$ homogenous Jordan polynomials  in the variables $x_1,\dots x_D$. The aim
of this paper is a conjecture about the character,
as a $GL(D)$-module, of each homogenous component $J_n(D)$ of 
$J(D)$. In the introduction,  only the conjecture for 
$\dim\,J_n(D)$ will be described, see Section 1.10 for the whole Conjecture 1.

\begin{conj1weakest} Set $a_n=\dim J_n(D)$. The sequence $a_n$ is the unique solution of  the following equation:

\noindent $({\cal E})$ \centerline{$
Res_{t=0}\, \psi\,\prod_{n}^{\infty} (1-z^n(t+t^{-1}) +z^{2n})^{a_n} dt=0$,}

\noindent where $\psi=Dzt^{-1}+(1-Dz)-t$.
\end{conj1weakest}

It is easy to see that equation ${\cal E}$ provides a recurrence relation to uniquely 
determine the integers $a_n$, but we do not know a closed formula.

Some computer 
calculations show that the predicted dimensions are
correct for some interesting cases. E.g., for
$D=3$ and $n=8$ the conjecture predicts that the space
of special identities has dimension 3, which is correct: those are the famous Glennie's Identities \cite{Gl}. Similarly for 
$D=4$  the conjecture
 agrees that  some  tetrads 
are missing in $J(4)$, as it has been observed by Cohn
\cite{Co}. Other interesting numerical evidences
are given in Section 2.
Since our input is the quite simple polynomial $\psi$, these numerical verifications provide a good support for the conjecture.

Conjecture 1 is elementary, but quite mysterious. 
Indeed it follows from two natural, but more sophisticated, conjectures
about Lie algebras cohomology.  Conjecture 3 will be now stated, see Section 3 for Conjecture 2.

Let ${\bf Lie_{T}}$  be the category of  Lie algebras $\fg$
on which $\fsl_2$ acts by derivation 
such that $\fg=\fg^{\fsl_2}\oplus \fg^{ad}$ as an $\fsl_2$-module.  
For  any Jordan algebra $J$, Tits has defined a Lie algebra
structure on the space $\fsl_2\otimes J\oplus \Inner\,J$
\cite{T}. It has been later generalized by  Kantor \cite{Kan}
and Koecher \cite{Koe}  and it is now called the $TKK$-construction and  denoted by $TTK(J)$. Here we use another refinement of Tits construction, due to Allison and Gao \cite{AG}. The corresponding Lie algebra will be denoted by $TAG(J)$, or, more simply, by $\fsl_2\,J$. 
The Lie algebra $\fsl_2\,J$ belongs to the category ${\bf Lie_{T}}$ 

 Since the $TAG$-construction is functorial, it is obvious that  
 $\fsl_2\,J(D)$ is a free 
Lie algebra in the category ${\bf Lie_{T}}$. Therefore  it is very natural to expect some cohomology vanishing, as the following

\begin{conj3} We have

\centerline{$H_k(\fsl_2\,J(D))^{\fsl_2}=0$, for any $k>0$.}
\end{conj3}

The conjecture, obvious for $k=1$,  follows from Allison-Gao paper \cite{AG} for  $k=2$. It is also proved for $k=3$ in Section 4. 
To connect the two conjectures, we first prove in Section 5

\begin{thm1imprecise}
The Jordan operad is cyclic.
\end{thm1imprecise}

This Theorem has some  striking consequences. E.g
the space ${\cal SI}(D)$ of multilinear special
identities of degree $D$, which is  is obviously a 
$\fS_D$-module, is indeed a $\fS_{D+1}$-module for any $D\geq 1$.
Also it allows to easily compute $\dim\,J_n(D)$ for
any $n\leq 7$, for any $D$.

In Section 6, we use a more technical version of Theorem 1 to prove that:

\begin{thm2}
Conjecture 3 implies Conjecture 1.
\end{thm2}

As a conclusion, the reader could find Conjecture 3  too optimistic.
However, it is clear from the paper that the
groups $H_*(\fsl_2\,J(D))^{\fsl_2}$ are strongly connected with 
the structure of the free Jordan algebras and they provide an interesting
approach for these questions.

\bigskip

\noindent {\it Acknowledgement.} We thank J. Germoni who performs some
of the computer computations of the paper.  We also would like to thank I. Shestakov for helpful discussions.
OM has been supported by UMR 5028 du CNRS. IK and OM have been supported by Cofecub Project 15716 
and the Udl-USP project "Free Jordan Algebras". IK has been supported by CNPQ 307998/2016-9.

\section*{1. Statement of Conjecture 1}

The introduction describes the {\it weakest version of Conjecture 1}, which determines the dimensions of the
homogenous components of $J(D)$. In this section,  Conjecture 1
will be stated, as well as a weak version of it.

Let  $\Inner\,J(D)$ be the Lie algebra of inner derivations of $J(D)$. 
Conjecture 1, stated in Section 1.9, provides  the character, as $GL(D)$-modules, of the
homogenous components of $J(D)$ and of $\Inner\,J(D)$.
The {\it weak version of Conjecture 1} is a formula only
for  the dimensions  of those homogenous components, see Section 1.10. 
 The Subsections 1.1 to 1.8 are devoted to define the main notations
of the paper, and to introduce the combinatorial notions which 
are required to state Conjecture 1.

\bigskip
\noindent{\it 1.1 Main notations and conventions} 

\noindent Throughout this paper, the ground field  $K$ has characteristic zero, and all algebras and vector spaces are
defined over $K$. 

Recall that a commutative algebra $J$ is called a  {\it Jordan algebra} if its product satisfies the following {\it Jordan identity}

\centerline{$x^{2}(yx)=(x^{2}y) x$} 

\noindent for any $x,y\in J $. For 
$x,\,y\in J$, let $\partial_{x,y}:J\rightarrow J$ be the map $z\mapsto x(zy)-(xz)y$. It follows from the Jordan identity that  $\partial_{x,y}$ is a derivation. A derivation 
$\partial$ of $J$ is called an {\it inner derivation} if it is
a linear combination of some $\partial_{x,y}$.
The space, denoted  $\Inner\, J$, of all inner derivations  is a  subalgebra of the Lie algebra $\Der\,J$ of all derivations of $J$.

In what follows, the positive integer $D$ will be given once for all. 
Let $J(D)$  be the free Jordan algebra on $D$ generators. This algebra,
and some variants, has been investigated in many papers by the 
Novosibirsk school of algebra, e.g. \cite{Med}, \cite{Med2}, \cite{RNA}, \cite {Z},\cite{Z2}.

\bigskip
\noindent {\it 1.2 The ring {\cal R(G)}}

\noindent Let $G$ be an algebraic reductive group
and let $Z\subset G$ be a central subgroup isomorphic to $K^{*}$. In what follows a rational $G$-module will be called a {\it $G$-module} or a {\it representation of $G$}.

Let $n\geq 0$. A $G$-module on which any
$z\in Z$ acts by $z^{n}$ is called a
{\it $G$-module of degree $n$}. Of course  this notion is relative to the the subgroup $Z$ and to the isomorphism
$Z\simeq K^{*}$. However we will assume that these data are given once for all.

Let $Rep_{n}(G)$ be the category of 
the finite dimensional $G$-modules of degree $n$. 
Set

 \centerline{${\cal R}(G)=\prod_{n=0}^{\infty} K_{0}(Rep_{n}(G))$}

\centerline{${\cal M}_{> n}(G)=\prod_{k >n} K_{0}(Rep_{k}(G))$}

\centerline{${\cal M}(G)={\cal M}_{> 0}(G)$.}

\noindent There are products  

\centerline{$K_{0}(Rep_{n}(G))\times K_{0}(Rep_{m}(G))\rightarrow  K_{0}(Rep_{n+m}(G))$}

\noindent induced by the
tensor  product  of the  $G$-modules. 
 Therefore ${\cal R}(G)$ is a ring
and ${\cal M}(G)$ is an ideal. 

Moreover 
${\cal R}(G)$ is complete with respect to the
${\cal M}(G)$-adic topology, i.e. the topology for which  the sequence   ${\cal M}_{> n}(G)$
is a basis of neighborhoods of 0. Any element $a$ of ${\cal R}(G)$ can be written as a formal series  

\centerline{$a=\sum_{n\geq 0}\, a_{n}$}

\noindent where $a_{n}\in K_{0}(Rep_{n}(G))$.

As usual, the class in $K_{0}(Rep_{n}(G))$ of a $G$-module 
$V\in Rep_{n}(G))$ is denoted by  $[V]$. Also
let $Rep(G)$ be the category of the $G$-modules $V$, with a decomposition
$V=\oplus_{n\geq 0}\,V_n$, such that  $V_n\in Rep_{n}(G)$ for all $n\geq 0$.
For such a module $V$,
its class $[V]\in {\cal R}(G)$ is defined by $[V]:=\sum_{n\geq 0}\,[V_n]$.

\bigskip
\noindent {\it 1.3 Analytic  representations of $GL(D)$ and their natural gradings}

\noindent A finite dimensional rational representation $\rho$ of $GL(D)$ is called {\it polynomial} if the map $g\mapsto \rho(g)$ is polynomial into the entries $g_{i,\,j}$ of the matrix $g$. The  center of $GL(D)$ is $Z=K^{*} \id$, relative to which the degree of a representation has been defined in the previous section. It is easy to show that  a polynomial representation $\rho$ has 
degree $n$ iff $\rho(g)$ is a degree $n$ homogenous polynomial into the entries $g_{i,\,j}$ of the matrix 
$g$. Therefore the notion of a polynomial representation of degree $n$ is unambiguously defined.

By definition an {\it analytic
$GL(D)$-module} is   
a $GL(D)$-module $V$  with a decomposition 

\centerline{$V=\oplus_{n\geq 0}\,V_n$} 

\noindent such that each component $V_n$ is a polynomial representation of degree $n$. In general $V$
is infinite dimensional, but it is always required that
each $V_n$ is finite dimensional. The decomposition
$V=\oplus_{n\geq 0}\,V_n$ of an analytic module $V$ is called
its {\it natural grading}.

The free Jordan algebra $J(D)$ and its associated Lie algebra $\Inner\,J(D)$ are examples of analytic $GL(D)$-modules. The natural grading of $J(D)$
is the previously defined decomposition $J(D)=\oplus_{n\geq 0}\,J_n(D)$
and the degree $n$ component of $\Inner\,J(D)$ is denoted 
$\Inner_nJ(D)$.

Let $Pol_n(GL(D))$ be the category of 
polynomial representations of $GL(D)$
of degree $n$, let $An(GL(D)$ be the category of all analytic  $GL(D)$-modules. Set 

\centerline{${\cal R}_{an}(GL(D))=\prod_{n\geq 0}\,K_0(Pol_n(GL(D)))$, and}

\centerline{${\cal M}_{an}(GL(D))=\prod_{n> 0}\,K_0(Pol_n(GL(D)))$.}

\noindent  The class $[V]\in {\cal R}_{an}(GL(D))$ of an analytic module is defined as before. 
 
 Similarly a finite dimensional rational representation $\rho$ of $GL(D)\times PSL(2)$ is called {\it polynomial} if the underlying 
 $GL(D)$-module is polynomial.
Also an {\it analytic
$GL(D)\times PSL(2)$-module} is   
a $GL(D)\times PSL(2)$-module $V$  with a decomposition 

\centerline{$V=\oplus_{n\geq 0}\,V_n$} 

\noindent such that each component $V_n$ is a polynomial representation of degree $n$.

\bigskip
\noindent {\it 1.4 Weights and Young diagrams}

\noindent
 Let $H\subset GL(D)$ be the subgroup of diagonal matrices.
 The subsection is devoted to the combinarics of
 the weights and the dominant weights of the polynomial
representations. 

A $D$-uple ${\bf m}=(m_1,\dots,\,m_D)$  of
non-negative integers  is called a {\it partition}.  It is called
a {\it partition of $n$} if 
$m_1+\dots+\,m_D=n$. The weight decomposition of an analytic
module $V$  is given by

\centerline{$V=\oplus_{\bf m}\, V_{\bf m}$}

\noindent where ${\bf m}$ runs over all the partitions, and where
$V_{\bf m}$ is the subspace of all $v\in V$
such $h.v=h_1^{m_1}h_2^{m_2}\dots\,h_D^{m_D}.v$ for all  $h\in H$ with diagonal entries $h_1, h_2,\dots\,h_D$. Relative to the natural grading
$V=\oplus_{n\geq 0}\, V_n$ of $V$, we have

\centerline{$V_n=\oplus_{\bf m}\, V_{\bf m}$}

\noindent where ${\bf m}$ runs over all the partition of $n$.

 With these notations, there is an isomorphism \cite{Macbook}
 
 \centerline
 {${\cal R}_{an}(GL(D))\simeq \Z[[z_1,\dots,z_D]]^{\fS_D}$}
 
 \noindent where the symmetric group $\fS_D$ acts by permutation of  the  variables $z_1,\dots,z_D$.
Then the class of an analytic module $V$ in 
${\cal R}_{an}(GL(D)$ is given by

\centerline{$[V]=\sum_{\bf m}\,
\dim V_{\bf m}\,\,z_1^{m1}z_2^{m_2}\dots z_D^{m_D}.$}
 
 \noindent For example, let
 $x_1,\dots,\,x_D$ be the generators of $J(D)$. Then for
 any partition
 ${\bf m}=(m_1,\dots,\,m_D)$,  $J_{\bf m}(D)$ is the space of Jordan polynomials $p(x_1,\dots,\,x_D)$ which are homogenous of degree $m_1$ into $x_1$, homogenous of degree $m_2$ into $x_2$
 and so on... Thus the class $[J(D)]\in {\cal R}_{an}(GL(D))$ encodes the same information as $\dim J_{\bf m}(D)$ for all ${\bf m}$.
 
Relative to the standard Borel subgroup, the dominant weights of polynomial representations are the partitions
${\bf m}=(m_1,\dots,\,m_D)$ with $m_1\geq m_2\geq\dots\geq m_D$
\cite{Macbook}. Such a partition, which is called a {\it Young diagram}, is represented by  a diagram with $m_1$ boxes on the first line, $m_2$ boxes on the second line and so on...  When a  pictorial notation is not convenient, it will be denoted as ${\bf (n_1^{a_1},n_2^{a_2}\dots)}$, where the symbol
 ${\bf n^a}$ means that the lign with $n$ boxes is repeated $a$ times. E.g.,  ${\bf (4^2,2)}$ is represented by
\medskip

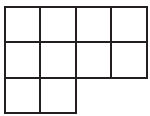
\centerline{$\yng(4,4,2)$}

For a Young diagram ${\bf Y}$,  the total number of boxes, 
namely $m_1+\dots+m_D$   is called its {\it size} while its {\it height} is the number of boxes on the first column. 
 When  ${\bf Y}$ has heigth $\leq D$, the simple $GL(D)$-module
 with highest weight ${\bf Y}$ will be denoted by $L({\bf Y};D)$. 
 It is also
 convenient to set $L({\bf Y};D)=0$ if the heigth of ${\bf Y}$ is $>D$.
 For example $L({\bf 1^3};D)$ denotes $\Lambda^3\,K^D$,  which is zero
 for $D<3$.

\bigskip
\noindent
{\it 1.5  Effective elements in 
${\cal R}(G)$.}

\noindent
The classes $[M]$ of the
$G$-modules $M$   are  called the {\it effective classes}
in ${\cal R}(G)$. Let  ${\cal M}(G)^{+}$ be the set of effective classes in ${\cal M}(G)$. Then any $a\in{\cal M}(G)$ can be written as
$a'-a"$, where $a',\,a"\in {\cal M}(G)^{+}$.

\bigskip
\noindent
{\it 1.6 $\lambda$-structure on the ring
${\cal R}(G)$} 

\noindent
The ring ${\cal R}(G)$ is endowed with a map
$\lambda:{\cal M}(G)\to{\cal R}(G)$.

First   $\lambda\,a$ is defined for $a\in {\cal M}^+(G)$. Any 
$a\in {\cal M}^+(G)$ is the class
of  a  $G$-module $V\in Rep(G)$.  It is clear that  $M:=\Lambda\,V$ 
belongs to $Rep(G)$. Set 

\centerline{$\lambda\, a=\sum_{k\geq 0}\,(-1)^k\,[\Lambda^k\,V]$.}

\noindent  Moreover we have $\lambda(a+b)=\lambda\,a \, \lambda\,b$
for any $a,\,b\in {\cal M}^+(G)$. 

For an arbitrary $a\in {\cal M}(G)$, there are
$a',\,a"\in {\cal M}^+(G)$ such that $a=a'-a"$. Since
$\lambda \,a" = 1$ modulo ${\cal M} (G)$, it is invertible,
and $\lambda\,a$ is defined  by

\centerline{$\lambda\,a= (\lambda\,a")^{-1} \lambda\,a'$.}

\bigskip
\noindent
{\it 1.7 The decomposition in the ring ${\cal R}(G\times PSL(2))$}

\noindent
Let $G$ be a reductive group.
For any $k\geq 0$, let $L(2k)$ be the irreducible $PSL(2)$-module of dimension $2k+1$.
Since the family  $([L(2k)])_{k\geq 0}$ is a basis
of $K_0(PSL(2))$, any element 
$a\in K_0(G\times PSL(2))$ can be writen as a finite sum

\centerline {
$a=\sum_{k\geq 0}\, [a:L(2k)]\, [L(2k)]$}

\noindent where the multiplicities $[a:L(2k)]$
are elements of $K_0(G)$.

Assume now that $G$ is a subgroup of  $GL(D)$
which contains the central subgroup $Z=K^*\id$.
We consider  $Z$ as a subgroup of  $G\times PSL(2)$,
and therefore the notion of  a $G\times PSL(2)$-module of degree $n$ is well defined. Indeed it means that the  underlying $G$-module has degree $n$.
As before any $a\in{\cal R}(G\times PSL(2))$ can be decomposed as 

\centerline {$a=\sum_{k\geq 0}\, [a:L(2k)]\, [L(2k)]$}

 \noindent where $[a:L(2k)]\in {\cal R}(G)$. Instead of being a finite sum, it is a series whose convegence comes from the fact that
 
 \centerline{
 $[a:L(2k)] \to 0$ when $k \to \infty$.}

\bigskip
\noindent{\it 1.8 The elements $A(D)$ and $B(D)$  in the ring 
${\cal R}_{an}(GL(D))$}

\noindent
Let $G\subset GL(D)$ be a reductive subgroup containing
$Z=K^*\id$. Let  $K^D$ be the natural representation of $GL(D)$ and let 
$K^D\vert_G$ be its restriction to $G$.

\begin{lemma}\label{combi}

1. There are elements $a(G)$ and $b(G)$ in
${\cal M}(G)$ which are uniquely defined by the following two equations in ${\cal R}(G\times PSL(2))$

\centerline
{$\lambda (a(G)[L(2)]+b(G)):[L(0)]=1$}

\centerline
{$\lambda (a(G)[L(2)]+b(G)):[L(2)]=-[K^D\vert_G].$}

2. For $G=GL(D)$, set $A(D)=a(GL(D))$ and 
$B(D)=b(GL(D))$. Then $A(D)$ and $B(D)$
are in ${\cal M}_{an}(GL(D))$.

3.  Moreover 
$a(G)=A(D)\vert_G$ and
$b(G)=B(D)\vert_G$. 

\end{lemma}

\begin{proof}

In order to prove Assertion 1, 
some elements  $a_n$ and 
$b_n$ in ${\cal M}(G)$
are defined by induction by the following algorithm.
Start  with $a_0=b_0=0$. Then assume that
 $a_n$ and $b_n$ are already defined with the property that

\centerline
{$\lambda (a_n[L(2)]+b_n):[L(0)]=1$\hskip11mm modulo ${\cal M}_{>n}(G)$}

\centerline
{$\lambda (a_n[L(2)]+b_n):[L(2)]=-[K^D\vert_G]$ \hskip2mm modulo ${\cal M}_{>n}(G)$.}

\noindent  Let $\alpha$ and $\beta$ be in
$K_0(Rep_{n+1}(G))$ defined by 

\centerline
{$\lambda (a_n[L(2)]+b_n):[L(0)]=1 -\alpha$
\hskip11mm modulo ${\cal M}_{>n+1}(G)$}

\centerline
{$\lambda (a_n[L(2)]+b_n):[L(2)]=-[K^D\vert_G]-\beta$
\hskip3mm modulo ${\cal M}_{>n+1}(G)$.}

\noindent 
Thus set $a_{n+1}=a_{n}+\alpha$ and $b_{n+1}=b_n+\beta$.
Since we have

\noindent $\lambda(\alpha[L(2)]+\beta)=1-\alpha.[L(2)]-\beta$ modulo ${\cal M}_{>n+1}(G)$,  we get

\centerline
{$\lambda (a_{n+1}[L(2)]+b_{n+1}):[L(0)]=1$
\hskip11mm modulo ${\cal M}_{>n+1}(G)$}

\centerline
{$\lambda (a_{n+1}[L(2)]+b_{n+1}):[L(2)]=-[K^D\vert_G]$
\hskip2mm modulo ${\cal M}_{>n+1}(G)$,}

\noindent   and therefore
the algorithm  can continue.

Since $a_{n+1}-a_n$ and $b_{n+1}-b_n$ belong to
$K_0(Rep_{n+1}(G))$, the sequences $a_n$ and $b_n$ converge.
The elements $a(G):=\lim a_n$ and $b(G):=\lim b_n$
satisfies the first assertion. Moreover, it is clear that $a(G)$ and 
$b(G)$ are uniquely defined.

The second assertion follows from the fact that,
for the group $G=GL(D)$, all calculations arise in the ring
${\cal R}_{an}(GL(D))$: so the elements 
$A(D)$ and $B(D)$ are in ${\cal M}_{an}(GL(D))$.

For Assertion 3, it is enough to notice that
the pair $(a(G),b(G))$ and $(A(D)\vert_G,B(D)\vert_G)$
satisfy  the same equation, so they are equal.
\end{proof}

\bigskip
\noindent
{\it 1.9 The conjecture 1}  

\noindent After these long preparations, we can now state Conjecture 1.

\begin{conj} Let $D\geq 1$ be an integer.
 In ${\cal R}_{an}(GL(D))$ we have

\centerline{$[J(D)]=A(D)$ and $[\Inner\,J(D)]=B(D)$,}

\noindent  where the elements $A(D)$ and $B(D)$  are defined in 
Lemma \ref{combi}.

\end{conj}

\bigskip
\noindent
{\it 1.10 The weak form of Conjecture 1}

\noindent
We will now state  the weak version of Conjecture 1 which  only involves  the dimensions of homogenous components of $J(D)$ and $\Inner\,J(D)$.

Here $G$ is the central subgroup $Z=K^* \id$ of 
$GL(D)$. As in the subsection 1.4, 
${\cal R}(Z)$ is identified  with $\Z[[z]]$. An 
$Z$-module $V\in Rep(G)$ is  a graded vector space
$V=\oplus_{n\geq 0}\, V_n$ and its class [V] is

\centerline{$[V]=\sum_n\,\dim\,V_n\,z^n.$}

\noindent Let $\alpha$ be a root of the Lie algebra $\fsl_2$ and set $t=e^{\alpha}$. Then $K_0(PSL(2))$ is the subring
 $\Z[t+t^{-1}]$ of $\Z[t,\,t^{-1}]$ consisting of the
 symmetric polynomials in $t$ and $t^{-1}$. If follows that 
 
 \centerline
 {${\cal R}(G\times PSL(2))=\Z[t+t^{-1}][[z]].$}

Next let $a\in K_0(PSL(2))$ and set
$a=\sum_i\, c_i\,t^i$. Since 
$[a:L(0)]=c_0-c_{-1}$ and
$[a:L(2)]=c_{-1}-c_{-2}$
\noindent it follows that 

\centerline{$[a:L(0)]=\Res_{t=0}\,(t^{-1}-1)a\,\d t$ and
$[a:L(2)]=\Res_{t=0}\,(1-t)a\,\d t$.}

\noindent Indeed the same formula holds when
$a$ and $b$ are in ${\cal R}(G\times PSL(2))$.
In this setting, Lemma \ref{combi} can be expressed as

\begin{lemma}\label{combi2} Let $D\geq 1$ be an integer.
There are two series
$a(z)=\sum_{n\geq n}\,a_n(D) z^n$ and 
$b(z)=\sum_{n\geq n}\,b_n(D) z^n$ in
$\Z[[z]]$ which are uniquely defined by the following two equations:

\centerline{$\Res_{t=0}\,(t^{-1}-1) \Phi\,\d t=1$}

\centerline{$\Res_{t=0}\,(1-t)\Phi\,\d t =-Dz$}

\noindent where
$\Phi=\prod_{n\geq 1}\, 
(1-z^n t)^{a_n} (1-z^n t^{-1})^{a_n}(1-z^n)^{a_n+b_n}$, 
$a_n=a_n(D)$ and $b_n=b_n(D)$.
\end{lemma}

 The weak version of Conjecture 1 is

\begin{conj1weak} Let $D\geq 1$. We have

\centerline{$\dim\,J_n(D)=a_n(D)$ and $\dim\,\Inner_n\,J(D)=b_n(D)$}

\noindent where $a_n(D)$ and $b_n(D)$ are defined in Lemma \ref{combi2}.

\end{conj1weak}

Indeed,   Lemma \ref{combi2} and the weak version of Conjecture 1 are the specialization of Lemma \ref{combi} and Conjecture 1 by the map
${\cal R}(GL(D)\times PSL(2))\to {\cal R}(Z\times PSL(2))$.

\bigskip
\noindent
{\it 1.11 About the weakest form of Conjecture 1} 

It is now shown that the version  of Conjecture 1, stated in the introduction,  is a consequence of the weak form of Conjecture 1.

\noindent It is easy to prove, as in Lemma \ref{combi},  that the series $a_n$ of the introduction is uniquely defined.
It remains to show that  the series $a_n$ of Lemma \ref{combi2} 
is the same.

Let's consider the series $a_n=a_n(D)$ and $b_n=b_n(D)$ of Lemma \ref{combi2}. 
We have

\centerline{$\Res_{t=0}\,(t^{-1}-1) \Phi\,\d t=1$, and}

\centerline{ 
$\Res_{t=0}\,(1-t)\Phi\,\d t =-Dz.$}

\noindent Using that the residue is $\Z[[z]]$-linear, and
combining the two equations we get
$\Res_{t=0}\,\psi\, \Phi\,\d t=0$, or, more explicitly

\centerline{$\Res_{t=0}\,\psi \prod_{n\geq 1}\,
(1-z^n)^{a_n+b_n} 
(1-z^n t)^{a_n} (1-z^n t^{-1})^{a_n}\,\d\,t=0$.}

\noindent By $\Z[[z]]$-linearity we  can remove the 
factor  $\prod_{n\geq 1}\,  (1-z^n)^{a_n+b_n}$ and so we get

\centerline{$\Res_{t=0}\,\psi \prod_{n\geq 1}\, 
(1-z^n t)^{a_n} (1-z^n t^{-1})^{a_n}\,\d\,t=0$}

\noindent which is the equation of the introduction.

\section*
{2. Numerical Evidences for Conjecture 1}

The numbers $\dim\,J_n(D)$ and 
$\dim\,\Inner_n\,J(D)$ are
known  in the following cases:

\smallskip
\noindent
\begin{tabular}{|l|c|c|c|c|}
 \hline
 $D$ & $\dim\,J_n(D)$ & Proof in & $\dim\,\Inner_n\,J(D)$ & Proof in \\
 \hline
  $D=1$ & any $n$ & folklore & any $n$ & folklore \\
\hline
   $D=2$ & any $n$ & Shirshov & any $n$ & Sect. 2.4 \\
\hline
   $D=3$ & $n\leq 8$ & Shirshov \& Glennie & $n\leq 8$ & Sect. 5\\
\hline
   $D$ any & $n\leq 7$ &  Sect. 5 & $n\leq 8$ & Sect. 5\\
\hline
\end{tabular}

\smallskip
\noindent The formulas for $\dim\,J_n(D)$, 
respectively for $\dim\,\Inner_n\,J(D)$,
are provided in Section 2.1, respectively in Section 2.2.  Then  we 
will describe for which cases Conjecture 1 has been checked.

\bigskip
\noindent{\it 2.1 General results about free Jordan algebras}

\noindent
Let $D\geq 1$ be an integer, let
 $T(D)$ be the non-unital tensor algebra on $D$ generators $x_1$, $x_2$, $\dots$, $x_D$. Let $\sigma$ be the involution on $T(D)$ defined by $\sigma(x_i)=x_i$.

 Given an associative algebra $A$, a subspace $J$ is called
a {\it Jordan subalgebra} if $J$ is stable by the Jordan product
$x\circ y=1/2 (xy+yx)$.
The Jordan subalgebra  $CJ(D)=T(D)^{\sigma}$ will be called the
{\it Cohn's Jordan algebra}. The Jordan subalgebra 
$SJ(D)$  generated by $x_1$, $x_2$, $\dots$, $x_D$ is called the {\it free special Jordan algebra}. The  kernel of the map 
 $J(D)\rightarrow CJ(D)$, which is denoted $SI(D)$, is
 called  the {\it space
  of  special identities}. Its cokernel
  $M(D)$ will be called  the {\it space of missing tetrads}.

The spaces $J(D)$, $T(D)$, $CJ(D)$, $SJ(D)$, $SI(D)$, $M(D)$
are all analytic $GL(D)$-modules. Relative to the natural grading,
the homogenous component of degree $n$  is respectively denoted by 
$J_n(D)$, $T_n(D)$, $CJ_n(D)$, $SJ_n(D)$, $SI_n(D)$, and
$M_n(D)$. There is
an exact sequence

\centerline{$0\to SI_n(D)\to J_n(D)
  \to CJ_n(D)\to M_n(D)\to 0$.}

\noindent Set $s_n(D)=\dim\,CJ_n(D)$.

\begin{lemma}\label{dimS}  We have

\centerline{$\dim\,J_n(D)=s_n(D)+\dim\,SI_n(D)-\dim\,M_n(D)$, where}

 \centerline{$s_{2n}(D)={1\over2}( D^{2n}+D^n), and$}
 
\centerline{  
 $s_{2n+1}(D)= {1\over2}( D^{2n+1}+D^{n+1})$}
 
 \noindent for any integer $n$.

\end{lemma}
  
\begin{proof} The first assertion comes from 
 the previous exact sequence.
The (obvious) computation of $\dim\,CJ_n(D)$ will be explained
 in Lemma \ref{dimInner}.
\end{proof}
 
 The previous elementary lemma shows that $\dim\,J_n(D)$
 is determined whenever $\dim\,SI_n(D)$ and $\dim\,M_n(D)$
 are known, as in the following three results:

 \begin{Glennietheorem} \cite{Gl}
  We have $SI_n(D)=0$ for $n\leq 7$.
 \end{Glennietheorem}

Let $t_4\in T(4)$  be the element

\centerline 
{$t_4= \sum_{\sigma\in \fS_4}\,\epsilon(\sigma)\,x_{\sigma_1}\dots x_{\sigma_4}$.}

\noindent Observe that $t_4$ belongs to $CJ(4)$.
Since $SJ(4)$ is commutative, it is clear that
$t_4(x_i,x_j,x_k,x_l)\notin SJ(D)$. 

\begin{Ctheorem} The Jordan algebra
$CJ(D)$ is generated by the elements
$x_1,x_2\dots x_D$  and $t_4(x_i,x_j,x_k,x_l)$
for all $1\leq i<j<k<l$.
In particular $M_n(D)=0$ for $D\leq 3$.
\end{Ctheorem}

The next lemma, together with Glennie Theorem, provides
an explicit formula for $\dim\,J_n(D)$ for any $n\leq 7$
and any $D\geq 1$.

\begin{lemma}\label{dimM} For any $D\geq 1$, we have have

$\dim M_4(D)= {D\choose 4}$, 

$\dim M_5(D)= D{D\choose 4}$, 

$\dim M_6(D)= 2 {D+1\choose 2}{D\choose 4}$,

 $\dim M_7(D)=2 D{D+1\choose 2}{D\choose 4}- \dim\,L({\bf 3,2,1^2};D)$. 
 
\end{lemma}

\begin{proof} 
The case $D=4$ follows from the fact that
$M_4(D)\simeq \Lambda^4\,K^D$.

By Corollary \ref{corM},  we have 
$M_5(D)=L({\bf 1^5}, D)\oplus L({\bf(2,1^3)},D)$
which  is isomorphic to $K^D\otimes \Lambda^4\,K^D$ and
the formula follows.  

It follows also from Corollary \ref{corM}  that 
$M_6(D)\simeq L({\bf 2,1^4};D)^2\oplus L({\bf 3,1^3}; D)^2$, 
which is isomorphic to 
$(S^2 K^D\otimes \Lambda^4 K^D)^2$. It follows also 
from the proof of Corollary \ref{corM} that, as a virtual module,
we have $[M_7(D)]=[K^D][M_6(D)]- [L({\bf 3,2,1^2};D)]$, what proves the formula. 

\end{proof}

\bigskip
\noindent{\it 2.2 General results about inner derivations}

\noindent

 Given an associative algebra $A$, a subspace $L$ is called
a {\it Lie subalgebra} if $L$ is stable by the Lie product
$[x,y]= (xy-yx)$.

\begin{lemma}\label{=[]}
Let $A$ be an associative algebra, let $Z(A)$ be its center 
and  let $J\subset A$ be a Jordan subalgebra. Assume $J$ contains a set of generators
of $A$ and that $Z(A)\cap [A,A]=0$. Then we have
$\Inner\,J\simeq [J,\,J]$.
\end{lemma}

\begin{proof} Set $C(J)=\{a\in A\vert [a,\,J]=0\}$.
As $J$ contains a set of generators of $A$, we have
$C(J)=Z(A)$.  Note that 
$4 \,\partial_{x,y}\,z= [[x,y],z]=\ad([x,y])(z)$ for any $x, y, z\in J$.
Since $[J,\,J]\cap C(J)=0$, we have 
$\Inner\,J=\ad([J,J])\simeq [J,J]$
\end{proof}

 For  $D\geq 1$, the space $A(D)=T(D)^{-\sigma}$ is a Lie subalgebra
 of $T(D)$.  By the previous
 lemma we have 
 
 \centerline{$\Inner\,SJ(D)=[SJ(D),SJ(D)]\subset
 \Inner\,CJ(D)=[CJ(D),CJ(D)]$.}
 
\noindent Therefore  $\Inner\,SJ(D)$ and $\Inner\,CJ(D)$
are Lie subalgebras of $A(D)$. There is a Lie algebra morphism

\centerline{$\Inner\, J(D)\to \Inner\,CJ(D)$.}

\noindent Its kernel $SD(D)$ will be called the {\it space
 of special derivations} and 
its  cokernel $MD(D)$ will be called the space
of {\it missing derivations}.

The spaces $\Inner\,J(D)$, $\Inner\,CJ(D)$, $\Inner\,SJ(D)$,
$A(D)$,  $SD(D)$ and $MD(D)$ are all analytic $GL(D)$-modules.
Relative to the natural grading,
the homogenous component of degree $n$  is respectively denoted by 
$\Inner_n\,J(D)$, $\Inner_n\,CJ(D)$, $\Inner_n\,SJ(D)$,
$A_n(D)$,  $SD_n(D)$ and $MD_n(D)$.

There is
an exact sequence

\centerline{$0\to SD_n(D)\to \Inner_n\,J(D)
  \to \Inner_n\,CJ(D)\to MD_n(D)\to 0$.}

Set $r_{n}(D)=\dim \Inner_n\,CJ(D)$.

 \begin{lemma}\label{dimInner} We have $\Inner\,CJ(D)=A(D)\cap [T(D),T(D)]$, and
 
 \centerline
 {$\dim\Inner\,J_n(D)=r_n(D)+\dim\,SD_n(D)-\dim\,MD_n(D)$, where}

\centerline{ $r_{2n}(D)=
{1\over 2}D^{2n}+{1\over 4}(D-1)D^n -
 {1\over 4n}\sum_{i|2n} \phi(i) D^{{2n\over i}}$,} 
 
\centerline{ $r_{2n+1}(D)={1\over 2}D^{2n+1}-
 {1\over 4n+2}\sum_{i|2n+1} \phi(i) D^{2n+1\over i}$.}

 \noindent for any $n\geq 1$, where $\phi$ is the Euler's totient function.
 
 \end{lemma}
 
 \begin{proof}
1) We have $[T(D),T(D)]=\sum_i\,[x_i,T(D)]$, so we get
 
\centerline{ 
$A(D)\cap [T(D),T(D)]=\sum_i\,[x_i,CJ(D)]\subset [CJ(D),CJ(D)]$.} 
 
\noindent Therefore we have $[CJ(D),CJ(D)]= A(D)\cap [T(D),T(D)]$, and it follows from Lemma \ref{=[]} that
$\Inner\,CJ(D)=A(D)\cap [T(D),T(D)]$.

2) Let $\sigma$ be an involution 
preserving  a basis $B$ of some vector space $V$.
An element $b\in B$ is called {\it oriented}
if $b\neq \sigma(b)$. Thus  $B$ is union of $B^{\sigma}$
and of its {\it oriented pairs} $\{b,b^\sigma\}$.
The following formulas will be used repeatively

$\dim\,V^\sigma={1\over 2}\,(\Card\, B+\Card\,B^\sigma),$ 

$\dim\,V^{-\sigma}$ is the number of oriented pairs.

3) The set of words in $x_1,\dots,x_D$ is a 
$\sigma$-invariant basis of $T(D)$, thus 
the formula for $s_n(D)$ (which was stated in Lemma \ref{dimS}) and for
$\dim\,A_n(D)$ follows  from the previous formulas.

4) A {\it cyclic word} is a word modulo cyclic permutation: for example $x_1 x_2 x_3$ and $x_2 x_3 x_1$ define the same cyclic word.  For $n,\,D\geq 1$, let $c_n(D)$ be the number of pairs of oriented 
cyclic words of length $n$ on a alphabet with $D$ letters.  E.g. $c_6(2)=1$ since $x^2y^2xy$ and $yxy^2 x^2$ is the only pair of oriented words of length $6$ in two letters.

 In the literature of Combinatorics, a cyclic word is often called a necklace while a non-oriented word is called a bracelet, and their enumeration is quite standard. There  are closed formulas for  both, 
 the webpage \cite{W} is nice. From this it follows that

$c_{2n}(D)={1\over 4n}\sum_{i|2n} \phi(i) D^{{2n\over i}}-{1\over 4} (D+1)D^{n}$, and

$c_{2n+1}(D)= {1\over 4n}\sum_{i|2n+1} \phi(i) D^{2n+1\over i} -{1\over 2} D^{n+1}$

\noindent for any $n\geq 1$, where $\phi$ denotes  the Euler's totient function.

Since the set of cyclic words is a basis of 
$T(D)/[T(D),T(D)]$, we have

\centerline{$\dim A_n(D)/[T(D),T(D)]\cap A_n(D)=c_n(D).$}
 
\noindent  Using the short exact sequence
 
\centerline{
$0\rightarrow \Inner(CJ(D)) \to A(D)\to A(D)/[T(D),T(D)]\cap A(D)\to 0$}

\noindent we get that  
$\dim \Inner_n\,CJ(D)=\dim A_n(D)-c_n(D)$ from which the explicit formula for $r_n(D)$ follows.
 \end{proof}

Since $\dim\Inner_n\,J(D)=r_n(D) +\dim\,SD_n(D) -
\dim\,MD_n(D)$, the next two lemmas compute
$\dim\Inner_n\,J(D)$ for any $n\leq 8$ and any $D\geq 1$.

\begin{lemma}\label{Innerleq8} We have 
$SD_n(D)=0$ for any $n\leq 8$ and any $D$.
\end{lemma}

\begin{proof}
The lemma  follows from Corollary \ref{Innerleq8} proved in Section 5.

\end{proof}

\begin{lemma}\label{dimMD} We have  $MD_n(D) =0$ for $n\leq 4$, and

$\dim\,MD_5(D)= D {D\choose 4} -{D\choose 5}$, 

$\dim\, MD_6(D)= {D\choose 6}+D^2 {D\choose 4}-D{D\choose 5}$,

$\dim MD_7(D)= 2 [D \dim L({\bf 3,1^3}; D) 
+{D\choose 2}{D\choose 5}-{D\choose 7}]$.

\noindent Moreover Corollary \ref{corMD} provides a (very long) formula for $\dim MD_8(D)$.

\end{lemma}

\begin{proof}

We have $[L({\bf 2,1^3};D)]=[K^D\otimes \Lambda^4\,K^D]-[\Lambda^5\,K^D]$.
Using Corollary \ref{corMD}, we have 
$\dim\,MD_5(D)=\dim\, L({\bf 2,1^3};D)=D {D\choose 4} -{D\choose 5}$.

By Corollary \ref{corMD} we have
$MD_6(D)\simeq L({\bf 1^6}; D)\oplus 
L({\bf 2,1^4})(D)\oplus L({\bf 2^2,1^2};D)
\oplus L({\bf 3,1^3 }; D)$ which is isomorphic to
$\Lambda^6 K^D\oplus K^D\otimes MD_5(D)$, from which the formula follows.

We have $K^D\otimes L({\bf 3,1^3}; D)=L({\bf 4,1^3}; D)
\oplus L({\bf 3,2,1^2}; D)  \oplus L({\bf 3,1^4}; D)$ and
$\Lambda^2\,K^D\otimes \Lambda^2\,K^D= L({\bf 2^2,1^3}; D)\oplus L({\bf 2,1^5}; D) \oplus \Lambda^5\,K^D$. 
By Corollary \ref{corMD}, $MD_7(D)$ is isomorphic to

$[L({\bf 4,1^3}; D)\oplus L({\bf 3,2,1^2}; D) 
\oplus L({\bf 3,1^4}; D)
\oplus L({\bf 2^2,1^3}; D)\oplus L({\bf 2,1^5}; D)]^2$

\noindent follows that, as a virtual module, we have
 
$[MD_7(D)]= 2 ([K^D]  [L({\bf 3,1^3}; D)]
+[\Lambda^2\,K^D] [\Lambda^2\,K^D]-[\Lambda^7\,K^D])$,

\noindent what proves the formula.

\end{proof}

\noindent
{\it 2.3 The case $D=1$}  

\noindent
Set $\Phi=\prod\limits_{n=1}^{\infty}(1-z^nt)(1-z^nt^{-1})(1-z^n)$ and, for any $n\geq 0$, set 
$P_n=t^{-n}+t^{-n+1}+\dots+t^n$. 

Observe that $\Res_{t=0} (t^{-1}-1)P_n dt$ is $1$ for $n=0$, and $0$ when $n>0$. Similarly we have
$\Res_{t=0} (1-t)P_n dt$ is $1$ when $n=1$ and $0$ otherwise. Using the classical Jacobi triple identity
\cite{HW}

\centerline
{$ \Phi =\sum_{n=0}^{\infty} (-1)^n z^{\frac{n(n+1)}{2}}
P_n$}

\noindent it follows that

\centerline
{$\Res_{t=0} (t^{-1}-1)\Phi dt=1$ and  
$\Res_{t=0} (1-t)\Phi dt=-z$.}

\noindent therefore we have $a_n(1)=1$ and $b_n(1)=0$ for any $n$.
This is in agreement with the fact that 
$J(1)=xK[x]$ and $\Inner\,J(1)=0$, and so Conjecture 1
holds for $D=1$.

\bigskip
\noindent
{\it 2.4 The case $D=2$} 

\noindent
Recall the following

\begin{Shtheorem} We have J(2)=CJ(2).
\end{Shtheorem}

Therefore, it is easy to compute $\dim\,J_n(2)$. Using that
$\dim\,\Inner_n\,J(2)=\dim\,A_n(2)-c_n(2)$, the computation
of $\dim\,\Inner_n\,J(2)$ is easily deduced from
the value of $c_n(2)$. These values are computed in
\cite{W} for $n\leq 15$ (it is the number $N(n,2) -N'(n,2)$  of \cite{W}). 
From which it has been
checked using a computer that $a_n(2)=s_n(2)=\dim J_n(2)$ and 
$b_n(2)=r_n(2)=\dim\Inner_n\, J(2)$ for any $n\leq 15$.

\bigskip
\noindent{\it 2.5 The case $D=3$} 

\noindent
In the case $D=3$,  recall the following

\begin{SCtheorem} The map $J(3)\to CJ(3)$ is onto.
\end{SCtheorem}

\begin{Mactheorem} The space $SI(3)$ contains no
Jordan polynomials of degree $\leq 1$ into $x_3$.
\end{Mactheorem}

The space $SI_8(3)$ contains the special identity $G_8$,
discovered in \cite{Gl}, which
is called the {\it Glennie's identity}. It is multi-homogenous
of degree $(3,3,2)$. In addition of the original
expression,   there are  two simpler formulas  due
to Thedy and Shestakov \cite{McC} and \cite{Sverch}.

\begin{GlennieIdentity} 1. We have $G_8\neq 0$. 

2.  $SI_8(3)$ is the $3$-dimensional $GL(3)$-module generated by $G_8$.
\end{GlennieIdentity}

Assertion 1 is proved in \cite{Gl}, and the fact $G_8$ generates
a $3$-dimensional  $GL(3)$-module is
implicite in \cite{Gl}.  However, we did not find a full proof of Assertion 2
 in the litterature, but the experts consider it as true. 
Note that by Macdonald's Theorem, no partition 
${\bf m}=(m_1,m_2,m_3)$ with $m_3\leq 1$ is a weight
of $SI(3)$. It seems to be known that $(4,2,2)$ is not
a weight of $SI_8(3)$ and  that the highest weight $(3,3,2)$ has multiplicity one.

It follows from Glennie's Theorem and Shirshov-Cohn's Theorem that

\centerline{$\dim \,J_n(3)=\dim SJ_n(3)=\dim\,CJ_n(3)=s_n(3)$ for 
$n\leq 7$,} 

\noindent while the previous 
argument shows (or suggests, if the reader do 
consider Assertion 2 as a conjecture) that $\dim \,J_8(3)=s_8(3)+3$ 

It follows from Lemma \ref{Innerleq8} and Shirshov-Cohn's Theorem that

\centerline{$\dim\,\Inner_n\,J(3)=\Inner_n\,SJ(3)=\dim\Inner_n\,CJ(3)=r_n(3)$ for $n\leq 8$.}

\noindent
This correlates with the computer data that
$a_n(3)=s_n(3)$  for $n\leq 7$ , while $a_8(3)=s_8(3)+3$
and similarly $b_n(3)=r_n(3)$ for $n\leq 8$.

\bigskip
\noindent
{\it 2.6 The case $D=4$} 

\noindent
By Glennie's Theorem we have $J_n(4)=SJ_n(4)$ for $n\leq 7$,
thus we have $\dim\,J_n(4)=s_n(4)-\dim\,M_n(4)$ for $n\leq 7$. 
Therefore it follows from Lemma \ref{dimM} that
$\dim\,J_n(4)=s_n(4)$, for $n\leq 3$, while
$\dim\,J_4(4)=s_4(4)-1$, 
$\dim\,J_5(4)=s_5(4)-4$,
$\dim\,J_6(4)=s_6(4)-20$,
and $\dim\,J_7(4)=s_7(4)-60$.

Similarly,  we have $\Inner_n\,J(4)=\Inner_n\,SJ_k(4)$
for $n\leq 8$
by Lemma \ref{Innerleq8} and therefore we have 
$\dim\,\Inner\,J_n(4)=r_n(4)-\dim\,MD_n(4)$ for $n\leq 8$. 
It follows from Lemma \ref{dimMD} that
$\dim\,\Inner_n\,J(D)=r_n(4)$, for $n\leq 4$ while
$\dim\,\Inner_5\,J(D)=r_5(4)-4$
$\dim\,\Inner_6\,J(D)=r_6(4)-16$ and 
$\dim\,\Inner_7\,J(D)=r_7(4)-80$.

 Some computer computations show  that
these dimensions agrees with the numbers $a_n(4)$ and $b_n(4)$
of Conjecture 1.

\bigskip
\noindent{\it 2.6 Conclusion}

\noindent
These numerical computations show that the Conjecture 1 takes into account the "erratic" formula for  $\dim\,\Inner\,J_n(2)$, also it
detects the special identities in $J_8(3)$ and the
missing  tetrads  in $J(4)$.  It provides some support for the conjecture, because these facts were not put artificially in Conjecture 1. Later on, we will see that  Conjecture 1 is also supported by some theoretical results and by the more theoretical conjectures 2 and 3.

\section*
{3. The Conjecture 2}

Conjecture 1 is an elementary statement, but it looks quite mysterious. In this section, the very natural, but less elementary,   Conjecture 2
will be stated. At the end of the section, it will be proved that Conjecture 2 implies Conjecture 1.

\bigskip
\noindent
{\it 3.1 The Tits functor $T:{\bf Lie_T}\rightarrow{\bf Jor}$.}

\noindent
Let ${\bf T}$ be the category of $PSL(2)$-modules $M$ such that $M=M^{\fsl_2}\oplus M^{ad}$, where $M^{ad}$ denotes the  isotypical component of $M$ of adjoint type.
Let ${\bf Lie_T}$ 
be the category of Lie algebras $\fg$ in category ${\bf T}$
on which $\fsl_2$ acts by derivation
(respectively endowed with an embedding 
$\fsl_2\subset\fg$). 

Let ${\bf Jor}$ 
be the category of Jordan algebras. Let $e$, $f$, $h$ be the usual basis 
of $\fsl_2$.  For  $\fg\in {\bf Lie_T}$, set

\centerline{
 $T(\fg)=\{x\in\fg\,|\,[h,x]=2x\}$.}
 
 \noindent Then $T(\fg)$ has an algebra structure, where the product 
 $x\circ y$ of any two elements $x,\,y\in T(
 \fg)$ is defined by:

\centerline{$
x\circ y=\frac12\,[x,f\cdot y]$.}

\noindent It turns out that $T(\fg)$ is a Jordan algebra
\cite{T}. So the map
$\fg \mapsto T(\fg)$ is a  functor
$T:{\bf Lie_T}\to {\bf Jor}$. 
It will be called the {\it Tits functor}.

\bigskip
\noindent
{\it 3.2 The TKK -construction}

\noindent
To each Jordan algebra $J$ is associated a Lie algebra
$TKK(J)\in {\bf Lie_T}$ which is defined as follows. As a vector space we have

\centerline{$TKK(J)=\Inner\,J \oplus \fsl_2\otimes J$.}

\noindent For $x\in \fsl_2$ and $a \in J$, set
$x(a)=x\otimes a$. The bracket $[X,Y]$ of two elements in $TKK(J)$ is defined as follows. When at least one
argument lies in $\Inner\,J $, it is defined by the fact that $\Inner\,J $ is a Lie algebra acting on $J$. Moreover 
 the bracket of two elements 
$x(a),\,y(b)$ in $\fsl_2\otimes J$ is given by

\centerline {$[x(a), y(b)]=[x,y](a\circ b) +
\kappa (x,y)\,\partial_{a,b}$}

\noindent where $\kappa$ is the invariant bilinear form on $\fsl_2$ normalized by the condition $\kappa(h,h)=4$.
This construction first appears in Tits paper \cite{T}. Later this definition has been generalized  by Koecher \cite{Koe} and Kantor \cite{Kan} in the theory of Jordan pairs (which is beyond the scope of this paper) and therefore the Lie algebra $TKK(J)$ is usually called the {\it TKK-construction}. 

However the notion of an inner derivation  is not functorial and therefore  the map  $J\in{\bf Jor}\mapsto TKK(J)\in{\bf Lie_T}$ is {\it not} functorial.

\bigskip
\noindent
{\it 3.3 The Lie algebra $TAG(J)=\fsl_2\,J$}

\noindent
More recently, Allison and Gao \cite {AG}
found another generalization (in the theory of structurable algebras) of Tits construction, see also \cite{Smirnov} and \cite{LM}. In the context of a Jordan algebra $J$, this
provides a refinement of the TKK-construction. The corresponding Lie
algebra will be called the {\it Tits-Allison-Gao construction} and
it will be denoted by  $TAG(J)$ or simply by $\fsl_2\,J$.

Let $J$ be any Jordan algebra. First  
$TAG(J)$ is defined as a vector space. Let $R(J)\subset 
\Lambda^2 J$ be the linear span of all $a\wedge a^2$ where $a$ runs over $J$ and  set  ${\cal B}J=\Lambda^2 J/R(J)$. Set 

\centerline{$TAG(J)={\cal B}J\oplus \fsl_2\otimes J$.}

Next, define the Lie algebra structure on $TAG(J)$.
 For  $\omega=\sum_i\,a_i\wedge b_i\in\Lambda^2\,J$, set
$\partial_{\omega}
=\sum_i\,\partial_{a_i,b_i}$ and let
$\{\omega\}$ be its image in ${\cal B}J$.
By Jordan identity we have
$\partial_{a,a^2}=0$, so there is a natural map

\centerline{
${\cal B}J \rightarrow \Inner J,\,
\{\omega\}\mapsto \partial_{\omega}$.}

\noindent Given another element  $\omega'=\sum_i\,a_i'\wedge b_i'$ in $\Lambda^2\,J$, set 
$\delta_{\omega}.\omega'=
\sum_i\,(\partial_\omega.a_i')\wedge b_i'
+ a_i'\wedge \partial_\omega. b_i'$. Since
$\partial_\omega$ is a derivation, we have
$\partial_\omega.R(J)\subset R(J)$ and therefore
we can set $\partial_\omega.\{\omega'\}=
\{\partial_\omega.\omega'\}$.

\noindent 
The bracket on $TAG(J)$ is defined by the following rules

1. $[x(a),y(b)]=[x,y](a\circ b)+\kappa(x,y) \{a\wedge b\}$, 

2. $[\{\omega\},x(a)]=x(\partial_{\omega} a)$, and

3. $[\{\omega\},\{\omega'\}]=\partial_{\omega}. \{\omega'\}$,

\noindent for any $x,y\in\fsl_2$, $a,b\in J$
and $\{\omega\},\{\omega'\}\in{\cal B}J$, where, as before
we denote  by  $x(a)$ the element $x\otimes a$ and where
$\kappa(x,y)={1\over 2} \,\Tr\, \ad(x)\circ \ad(y)$.

It is proved in \cite{AG} that $TAG(J)$ is a Lie algebra (indeed the tricky part is the proof that $[\{ \omega\},\{\omega'\}]$ is skew-symmetric). In general
$TKK(J)$ and $TAG(J)$ are different. For  $J=K[t,t^{-1}],$ we have $\Inner (J)=0$,
while ${\cal B}J$ is a one-dimensional Lie algebra. 
Therefore $TKK(J)=\fsl_2(K[t,t^{-1}])$ while
$TAG(J)$ is the famous affine Kac-Moody Lie algebra 
$\widehat{\fsl_2}(K[t,t^{-1}])$.

\begin{lemma}\label{ucover} Let $\fg\in{\bf Lie_T}$. Then there is
a Lie algebra morphism

\centerline{$\theta_\fg:TAG(T(\fg))\rightarrow \fg$}

\noindent which is the identity on $T(\fg)$.

\end{lemma}

\begin{proof}
Set $\fd=\fg^{\fsl_2}$, so we have
$\fg=\fd\oplus \fsl_2\otimes T(\fg)$. Since
$\Hom_{\fsl_2}(\fsl_2^{\otimes 2},K)=K.\kappa$,
there is a bilinear map
$\psi:\Lambda^2T(\fg)\rightarrow \fd$ such that

\centerline
{$[x(a),y(b)]=[x,y](a\circ b)+\kappa(x,y)\,\psi(a,b)$}

\noindent
for any $x,\,y\in\fsl_2$ and $a,\,b\in J$. 
For $x,\,y\,,z\in \fsl_2$, we have

\centerline
{$[x(a),[y(a),\,z(a)]]=[x,\,[y,\,z]](a^3)+
\kappa(x,[y,z])\,\psi(a,a^2)$.}

The map $(x,y,z)\mapsto \kappa(x,[y,z])$ has a cyclic symmetry of order $3$. Since $\kappa(h,[e,f])=4\neq 0$, the Jacobi identity for 
the triple 
$h(a),\,e(a),\,f(a)$ implies that 

\centerline{$\psi(a,a^2)=0$ for any $a\in J$.}

\noindent Therefore the map $\psi:\Lambda^2T(\fg)\rightarrow \fd$ factors trough ${\cal B}T(\fg)$.
A linear map $\theta_{\fg}:TAG(T(\fg))\rightarrow \fg$ is defined
by requiring that $\theta_{\fg}$ is the identity on 
$\fsl_2\otimes T(\fg)$ and $\theta_{\fg}=\psi$ on ${\cal B}T(\fg)$. It is easy to check that $\theta_{\fg}$ is a morphism of Lie algebras.

\end{proof}

It is clear that the map $TAG:J\in {\bf Jor}\mapsto TAG(J)\in{\bf Lie_T}$ is a functor, and more precisely we have:

\begin{lemma}\label{adjointfunctor}

The functor $TAG:{\bf Jor}\rightarrow{\bf Lie_T}$ is the left adjoint of the Tits functor $T$, namely:

\centerline{
$\Hom_{\bf Lie_T} (TAG(J),\fg)=
\Hom_{\bf Jor}(J,T(\fg))$}
\noindent
for any  $J\in {\bf Jor}$  anf $\fg\in{\bf Lie_T}$.
\end{lemma}

\begin{proof}
Let $J\in {\bf Jor}$  and $\fg\in{\bf Lie_T}$.
Since $T(TAG(J))=J$, any morphism of Lie algebra
$TAG(J)\rightarrow\fg$ restricts to a morphism of Jordan algebras $J\rightarrow T(\fg)$, so we there is a natural map

\centerline{$\mu: \Hom_{\bf Lie_T} (TAG(J),\fg)\rightarrow
\Hom_{\bf Jor}(J,T(\fg))$.}

Since the Lie algebra $TAG(J)$ is generated by 
$\fsl_2\otimes J$, it is clear that $\mu$ is injective.
Let  $\phi:J\rightarrow T(\fg)$ be a morphism of Jordan algebras. By functoriality of the $TAG$-construction, we get a Lie algebra morphism

\centerline{$TAG(\phi):TAG(J)\rightarrow TAG(T(\fg))$}

\noindent  and by Lemma \ref{ucover} there is a canonical
Lie algebra morphism

\centerline{$\theta_{\fg}:TAG(T(\fg))\rightarrow \fg$.}

\noindent So 
$\theta_{\fg}\circ TAG(\phi):TAG(J)\to\fg$ 
extends $\phi$ to a morphism of Lie algebras.
Therefore $\mu$ is bijective.

\end{proof}

\bigskip
\noindent
{\it 3.4 Statement of Conjecture 2}

\noindent
Let $D\geq 1$ be an integer and let $J(D)$ be the 
free Jordan algebra on $D$ generators.

\begin{lemma}\label{freeTAG}
The Lie algebra $\fsl_2\,J(D)$ is  free  in
the category ${\bf Lie_T}$.
\end{lemma}

The lemma follows from Lemma \ref{adjointfunctor} and the formal properties  of the adjoint functors. 

Let $k$ be a non-negative integer. Since   $\Lambda^k\,\fsl_2\,J(D)$ is a direct sum of $\fsl_2$-isotypical components of type $L(0),\,L(2),\dots,L(2k)$ there is a similar isotypcal decomposition of $H_k(\fg)$.
For an ordinary free Lie algebra
$\fm$, we have $H_k(\fm)=0$ for any $k\geq 2$. Here 
$\fsl_2\,J(D)$ is  free relative to category ${\bf Lie_T}$.
Since only the trivial and adjoint $\fsl_2$-type occurs in
the category  ${\bf T}$, the following conjecture seems very natural

\begin{conj}
We have

\centerline{$H_k(\fsl_2\,J(D))^{\fsl_2}=0$, and }

\centerline{$H_k(\fsl_2\,J(D))^{ad}=0$,}

\noindent for any $k\geq 1$.
\end {conj}

\bigskip

\noindent
{\it 3.5 Conjecture 2 implies Conjecture 1}

\begin{lemma}\label{nocenter}
Assume that $H_k(\fsl_2\,J(D))^{\fsl_2}=0$ for any odd $k$. Then 
we have ${\cal B}J(D)=\Inner\,J(D)$.
\end{lemma}

\begin{proof} Assume otherwise, i.e. assume that the
natural map $\phi:{\cal B}J(D)\rightarrow\Inner\,J(D)$
is not injective. Since ${\cal B}J(D)$ and $\Inner\,J(D)$
are analytic $GL(D)$-modules, they are endowed with the natural grading.
Let $z$ be a non-zero homogenous element
$z\in\Ker\,\phi$ and let $n$ be its degree. Set 
$G=\fsl_2\,J(D)/K.z$.
Since $z$ is a homogenous $\fsl_2$-invariant central element, 
$G$ inherits a structure of $\Z$-graded Lie algebra.

Moreover
$z$ belongs to $[\fsl_2\,J(D),\fsl_2\,J(D)]$. 
Therefore $\fsl_2\,J(D)$ is a non-trivial
central extension of $G$. Let 
$c\in H^2(G)$ be the corresponding cohomology class and
let $\omega\in (\Lambda^2\,G)^*$ be a homogenous two-cocycle representing $c$. We have
$\omega(G_i\wedge G_j)=0$ whenever
$i+j\neq n$. It follows that the bilinear map $\omega $ has finite  rank, therefore there exists an integer $N\geq 1$
such that $c^N\neq 0$ but $c^{N+1}=0$.

There is a long exact sequence of cohomology groups
\cite{HS}

\centerline
{$\dots H^k(G)
\buildrel j^* \over \longrightarrow
H^k(\fsl_2\,J(D)) 
 \buildrel i_z \over \longrightarrow 
 H^{k-1}(G)
 \buildrel \wedge c \over \longrightarrow 
 H^{k+1}(G)
 \buildrel j^* \over \longrightarrow \dots$
}

\noindent where $j^*$ is induced by the natural map  
$j:\fsl_2\,J(D)\rightarrow G$, where $i_z$ is the contraction by $z$ and where $\wedge c$ is the mutiplication by $c$. Therefore there exists
$C\in H^{2N+1}(\fsl_2\,J(D))$ such that $c^N=i_z\,C$. Since
$c^N$ is $\fsl_2$-invariant, we can assume that $C$ is 
also $\fsl_2$-invariant, and therefore 

\centerline {$H^{2N+1}(\fsl_2\,J(D))^{\fsl_2}\neq 0$}

\noindent which contradicts the hypothesis.

\end{proof}

\begin{cor}\label{2=>1}
Conjecture 2 implies Conjecture 1.
\end{cor}

\begin{proof} Assume Conjecture 2 holds.
In ${\cal R}_{an}(GL(D)\times PSL(2))$, the identity

\centerline{$[\Lambda^{even}\fsl_2\,J(D)]-[\Lambda^{odd}\fsl_2\,J(D)]=[H_{even}(\fsl_2\,J(D))]-[H_{even}(\fsl_2\,J(D))]$}

\noindent is Euler's characteristic formula.
By definition of the $\lambda$-operation, we have
$[\Lambda^{even}\fsl_2\,J(D)]-[\Lambda^{odd}\fsl_2\,J(D)]
=\lambda([\fsl_2\,J(D)])$. Moreover by
Lemma \ref{nocenter}, we have $[\fsl_2\,J(D)]=[J(D)\otimes L(2)]+
[\Inner\,J(D)]$, therefore we get

\centerline{
$\lambda([J(D)\otimes L(2)]+[\Inner\,J(D)])=
[H_{even}(\fsl_2\,J(D))]-[H_{even}(\fsl_2\,J(D))]$.}

\noindent It is clear that  $H_0(\fsl_2\,J(D))=K$ and

\centerline{$H_1(\fsl_2\,J(D))=\fsl_2\,J(D)/[\fsl_2\,J(D),\fsl_2\,J(D)]
\simeq K^D\otimes L(2)$.}

\noindent Moreover, by
hypothesis, the higher homology groups $H_k(\fsl_2\,J(D))$ contains no trivial or adjoint component. It follows that

\hskip21mm $\lambda([J(D)\otimes L(2)]+[\Inner\,J(D)]):[L(0)]=1$, and

\hskip21mm $\lambda([J(D)\otimes L(2)]+[\Inner\,J(D)]):[L(2)]=-[K^D]$.

\noindent
So by Lemma \ref{combi}, we get $[J(D)]=A(D)$ and 
$[\Inner\,J(D)]=B(D)$.
\end{proof}

\section*
{4. Proved Cases of Conjecture 2}

This section shows three results supporting Conjecture 2:

1. The conjecture holds for $D=1$,

2. As a $\fsl_2$-module, $H_2(\fsl_2\,J(D))$ is 
isotypical of type $L(4)$, and

3. The trivial component of the $\fsl_2$-module $H_3(\fsl_2\,J(D))$ is trivial.

\bigskip
\noindent
{\it 4.1 The $D=1$ case}

\begin{prop} Conjecture 2 holds for $J(1)$.
\end{prop}

For $D=1$, we have $J(1)=tK[t]$. So
Conjecture 2 is an obvious consequence of the following :

\begin{GLtheorem}\cite{GL} For any $k\geq 0$, we have

\centerline{$H_k(\fsl_2(tk[t]))\simeq L(2k)$.}
\end{GLtheorem}

Conversely, Garland-Lepowski Theorem can be used to prove that $J(1)=tK[t]$. Of course,  it is a complicated proof of a very simple result!

\bigskip
\noindent
{\it 4.2 Isotypical components of $H_2(\fsl_2\,J(D))$.}

\noindent
Let $D\geq 1$ be an integer. 
Let  ${\cal M}_{\bf T}(\fsl_2\,J(D))$ be the category
of $\fsl_2\,J(D)$-modules in category ${\bf T}$.
As an analytic $GL(D)$-module, $\fsl_2\,J(D)$ is endowed with the
natural grading.
Let ${\cal M}_{\bf T}^{gr}( \fsl_2\,J(D))$ be the category of 
all $\Z$-graded $\fsl_2\,J(D)$-modules
 $M\in{\cal M}_{\bf T}(\fsl_2\,J(D))$  such that

1. $\dim\,M_n<\infty$ for any $n$,  and 

2. $M_n=0$ for $n\gg 0$.

\begin{lemma}\label{techvan} Let $M$ be a $\fsl_2\,J(D)$-module. 
Assume that 

1. $M$ belongs to ${\cal M}_{\bf T}(\fsl_2\,J(D))$ and $\dim M<\infty$, or

2. $M$ belongs to ${\cal M}_{\bf T}^{gr}( \fsl_2\,J(D))$

Then we have $H_2(\fsl_2\,J(D),M)^{\fsl_2}=0$.
\end{lemma}

\begin{proof}
1) First assume $M$ belongs to ${\cal M}(\fsl_2\,J(D))$ and $\dim M<\infty$.
Let $c\in H^2(\fsl_2\,J(D),M^*)^{\fsl_2}$. Since $\fsl_2$ acts reductively, $c$ is represented by a $\fsl_2$-invariant cocycle 
$\omega:\Lambda^2\,\fsl_2\,J(D)\rightarrow M^*$. This cocycle
defines a Lie algebra structure on $L:=M^*\oplus\fsl_2\,J(D)$. 
Let 

\centerline{$0\to M^*\to L\to \fsl_2\,J(D)\rightarrow 0$}

\noindent be the corresponding abelian extension of
$\fsl_2\,J(D)$. Since $\omega$ is $\fsl_2$-invariant, it follows that 
$L$ lies in ${\bf Lie_T}$. By Lemma \ref{freeTAG} $\fsl_2\,J(D)$ is free in this category, hence the previous abelian extension is trivial. Therefore we have

\centerline{$H^2(\fsl_2\,J(D),M^*)^{\fsl_2}=0$.}

\noindent  By duality, it follows that $H_2(\fsl_2\,J(D),M)^{\fsl_2}=0$.

2) Assume now that $M$ belongs to ${\cal M}_{T}^{gr}( \fsl_2\,J(D))$. The 
$\Z$-gradings of $\fsl_2\,J(D)$ and $M$ induce  a grading of 
$H_*(\fsl_2\,J(D),M)$.  Relative to it,  the  degree $n$ component is denoted by  $H_*(\fsl_2\,J(D),M)\vert n$
and its $\fsl_2$-invariant part will be denoted by
$H^0(\fsl_2, H_2(\fsl_2\,J(D),M)\vert n$.

For an integer $n$,   set $M_{>n}=\oplus_{k>n}\,M_k$. Since 
the degree $n$-part of the 
complex $\Lambda\,\fsl_2\,J(D)\otimes M_{>n}$ is zero, we have

\centerline{$H_*(\fsl_2\,J(D),M)\vert n=H_*(\fsl_2\,J(D),M/M_{>n})\vert n$.} 

\noindent Since $M/M_{>n}$ is finite dimensional, the first part of the lemma shows that
$H^0(\fsl_2, H_2(\fsl_2\,J(D),M)\vert n=0$. Since $n$ is arbitrary,   we have 

\centerline{$H_2(\fsl_2\,J(D),M)^{\fsl_2}=0$.}
\end{proof}

\begin{prop}
The $\fsl_2$-module $H_2(\fsl_2\,J(D))$ is isotypical of
type $L(4)$.
\end{prop}

\begin{proof}  
It follows from Lemma \ref{techvan} that 
$H_2(\fsl_2\,J(D))^{\fsl_2}=0$.

The $PSL(2)$-module $L(2)$, with a trivial action of $\fsl_2\,J(D)$, belongs to ${\cal M}_{\bf T}(\fsl_2\,J(D))$. So 
it follows from Lemma \ref{techvan} that
$H_2(\fsl_2\,J(D), L(2))^{\fsl_2}=0$. Since

\centerline{$H_2(\fsl_2\,J(D))^{ad}=H_2(\fsl_2\,J(D), L(2))^{\fsl_2}\otimes L(2)$} 

\noindent we also have  $H_2(\fsl_2\,J(D))^{ad}=0$.

The only $PSL(2)$-types occurring in  $\Lambda^2\,\fsl_2\,J(D)$ are $L(0)$,
$L(2)$ and $L(4)$. Since the $L(0)$ and $L(2)$ types do not occur in $H_2(\fsl_2\,J(D))$, it follows that 
$H_2(\fsl_2\,J(D))$ is isotypical of
type $L(4)$.
\end{proof}

\bigskip
\noindent
{\it 4.3 Analytic functors}

\noindent
Let $Vect_K$ be the category of $K$-vector spaces and
let $Vect_K^{f}$ be the subcategory of finite dimensional
vector spaces. A functor $F:Vect_K\rightarrow Vect_K$
is called a {\it polynomial functor}
\cite{Macbook} if

1. $F(Vect_K^{f})\subset  Vect_K^{f}$ and $F$ commutes 
with the inductive limits,

2. There is some integer $n$ such that the map 

\centerline{$F:\Hom(U,V)\to\Hom(F(U),F(V))$} 

\noindent is a polynomial of degree $\leq n$ for any 
$U,\,V\in\,Vect_K^f$.
The polynomial functor $F$ is called 
{\it a  polynomial functor  of degree n} if 
$F(z\,\id_V)=z^n\,\id_{F(V)}$ for any
$V\in\,Vect_K^f$. It follows easily that
$F(V)$ is a polynomial $GL(V)$-module of degree $n$,
see \cite{Macbook}. Any polynomial
functor can be decomposed as a finite sum
$F=\oplus_{n\geq 0}\,F_n$,
where $F_n$ is a polynomial functor of 
degree $n$. 

A functor $F:Vect_K\to Vect_K$ is called
{\it analytic} if $F$ can be decomposed as a infinite sum

\centerline{$F=\oplus_{n\geq 0}\,F_n$}

\noindent where each $F_n$ is a polynomial functor of 
degree $n$. 
For an analytic functor $F$, it is convenient to set
$F(D)=F(K^D)$. For example, for $V\in Vect_K$, let $J(V)$ be
the free Jordan algebra  generated by the vector space $V$. Then 
$V\mapsto J(V)$ is an analytic functor, and 
$J(D)$ is the previously defined free Jordan algebra on $D$ generators.

\bigskip
\noindent
{\it 4.4 Suspensions of analytic functors.}

\noindent
Let $D\geq 0$ be an integer.  Let $K^D$ be  the space
with basis $x_1,\,x_2\dots x_D$.  To emphasize
the choice of $x_0$ as an additional vector, the vector space with basis $x_0,\,x_1\dots x_D$ will
be denoted by $K^{1+D}$ and its linear group will be denoted by $GL(1+D)$.

\begin{lemma}\label{degree1nonvan}
Let $M$ be an analytic $GL(1+D)$-module. Let 
${\bf m}=(m_0, \dots m_D)$ be a partition of some positive 
integer  such that 

\centerline{$M_{\bf m}\neq 0$ and $m_0=0$.}

\noindent Then there exists a partition 
${\bf m'}=(m'_0, \dots m_D')$  such that 

\centerline{ $M_{\bf m'}\neq 0$ and $m'_0=1$.}
\end{lemma}

\begin{proof}

By hypotheses, there is an index  $k\neq 0$ such that $m_k\neq 0$. Let
 $(e_{i,j})_{0\leq i,j\leq D}$  the usual basis of
$\fgl(1+D)$. Set $f=e_{0,k}$, $e=e_{k, 0}$ and
$h=[e,f]$. Then $(e,\,f,\,h)$ is a  $\fsl_2$-triple
in $\fgl(1+D)$. Let  ${\bf m}'$ be the partition of $n$ defined  by $m'_i=m_i$ if $i\neq 0$ or $k$,  $m'_k=m_k-1$
and $m_{0}'=1$. The eigenvalue of $h$ on
$M_{\bf m}$ is the negative integer $-m_k$, so the map
$e:M_{\bf m}\rightarrow M_{\bf m'}$ is injective,
and therefore $M_{\bf m'}$ is not zero.
\end{proof}

Let $F$ be an analytic functor. In what follows it will be convenient to
denote by $K.x_0$ the one-dimensional vector space with basis $x_0$.  Let $V\in Vect_K$.
For $z\in K^*$,  the element  $h(z)\in GL(K.x_0\oplus V)$
is defined by  $h(z).x_0=z\,x_0$ and $h(z).v=v$ for $v\in V$.
There is a decomposition

\centerline{$F(K.x_0\oplus V)=\oplus_{n}\, F(K.x_0\oplus V)\vert_n$} 

\noindent where 
$F(k\oplus V)\vert_n=\{v\in F(K.x_0\oplus V)\vert F(h(z)).v=
z^n v\}$. It is easy to see that 
$F(V)=F(K.x_0\oplus V)\vert_0$.
By definition, the {\it suspension} $\Sigma F$
of $F$ is the functor $V\mapsto F(K.x_0\oplus V)_1$. 
A functor $F$ is {\it constant} if $F(V)=F(0)$ for any 
$V\in Vect_K$.

\begin{lemma}\label{SigmaF=0} 1. Let $F$ be an analytic functor.
If $\Sigma F=\{0\}$, then $F$ is constant.

2. Let $F,\,G$ be two analytic functors with 
$F(0)=G(0)=\{0\}$, and let $\Theta:F\to G$ be a
natural transformation. If $\Sigma\Theta$ is an isomorphism,
then $\Theta$ is an isomorphism.

\end{lemma}

\begin{proof} 1) Let $F$ be a non-constant analytic functor.
Then for some integer $D$, there is a partition
${\bf m}=(m_1, \dots m_D)$ 
 of a positive integer such that 
$F(D)_{\bf m}\neq 0$. By lemma \ref{degree1nonvan}, there exist a 
partition ${\bf m'}=(m'_0, \dots m_D')$  with $m'_0=1$ such that
$F(1+D)_{\bf m'}\neq 0$. Therefore we have $\Sigma F(D)\neq 0$,
what proves the first assertion.

2) By hypothese have $\Sigma\Ker\Theta=\{0\}$ and  
$\Ker\Theta(0)=\{0\}$
(respectively $\Sigma\Coker\Theta\{0\}$ and $\Coker\Theta(0)=\{0\}$). It follows from the first assertion that $\Ker\Theta=\{0\}$ and
$\Coker\Theta=\{0\}$, therefore $\Theta$ is an isomorphism.

\end{proof}

\bigskip
\noindent
{\it 4.5 Vanishing of $H_3(\fsl_2\,J(D))^{\fsl_2}$}

\begin{prop} We have

\centerline{$H_3(\fsl_2\,J(D))^{\fsl_2}=0$.}
\end{prop}

\begin{proof}
We have 

\centerline{$\Sigma\,\Lambda \fsl_2\,J(D)= 
\Lambda\, \fsl_2\,J(D)\otimes \Sigma \fsl_2\,J(D)$.}

\noindent It follows that 
$\Sigma\,\Lambda \fsl_2\,J(D)$ is the complex computing 
the homology of $\fsl_2\,J(D)$ with value in the
$\fsl_2\,J(D)$-module $\Sigma\,\fsl_2\,J(D)$. Taking into account the degree shift, it follows that

\centerline{$\Sigma H_3(\fsl_2\,J(D))^{\fsl_2}=
H_2(\fsl_2\,J(D), \Sigma\,\fsl_2\,J(D))^{\fsl_2}$.}

\noindent Since $\Sigma\,\fsl_2\,J(D)$ belongs to ${\cal M}_{\bf T}^{gr}( \fsl_2\,J(D))$, it follows from Lemma \ref{techvan} that
$\Sigma H_3(\fsl_2\,J(D))^{\fsl_2}=0$. It follows from
Lemma \ref{SigmaF=0} that $H_3(\fsl_2\,J(D))^{\fsl_2}=0$.

\end{proof}

\section*
{5. Cyclicity of the Jordan Operads}

In this section, we  will  prove  that the  Jordan operad ${\cal J}$ is cyclic, what will be used in the last Section to simplify 
Conjecture 2. Also there are compatible cyclic structures on
the special Jordan operad ${\cal SJ}$ and the Cohn's Jordan operad
${\cal CJ}$. As a consequence, the degree $D$  multilinear space of special identies or missing tetrads are acted by $\fS_{D+1}$.

\bigskip
\noindent
{\it 5.1 Cyclic Analytic  Functors}

\noindent An analytic functor $F$ is called 
{\it cyclic} if $F$ is the suspension of some
analytic functor $G$. 
We will now describe a practical way to check
that an analytic functor is cyclic. 
In what follows, we denote by  $x_1,\dots x_D$
a basis of $K^D$ and we denote by $K^{1+D}$ the vector
space $K.x_0\oplus K^D$.

Let $F$, $G$ be two analytic functors and let 
$\Theta:  F\otimes \Id\to G$ be a natural transform,
where $\Id$ is the identity functor. Note that

\centerline{$\Sigma(F\otimes\Id)(D)=\Sigma F(D)\otimes K^D
\oplus F(D)\otimes x_0$.}

\noindent 
The triple $(F,G,\Theta)$ will be  called a 
{\it cyclic triple} if  the induced map

\centerline{$\Sigma F(D)\otimes K^D\to \Sigma G(D)$}

\noindent is an isomorphism, for any integer $D\geq 0$.

\begin{lemma}\label{crit}
Let  $(F,G,\Theta)$  be a 
cyclic triple. There is a natural isomorphism

\centerline{$F\simeq \Sigma\Ker \,\Theta$.}

\noindent In particular, $F$ is cyclic.  

 \end{lemma}

 \begin{proof} 
 We have 
 
\hskip21mm {$\Sigma(F\otimes \Id)(D)
=\Sigma F(D)\otimes K^D\oplus F(D)\otimes x_0$, and}
 
\hskip21mm{$\Sigma(F\otimes \Id)(D)
= \Sigma F(D)\otimes K^D\oplus \Ker\Sigma\Theta(D)$.}

\noindent Therefore $F(D)\simeq F(D)\otimes x_0$ is
naturally identified
 with $\Ker\,\Sigma\Theta(D)$, i.e. the functor 
$F$ is isomorphic to $\Sigma\Ker\,\Theta$. Therefore $F$ is cyclic.
\end{proof}

\bigskip
\noindent
{\it 5.2 $\fS$-modules}

\noindent
Let $D\geq 1$. For any Young diagram ${\bf Y}$ of size $D$, let 
${\bf S}({\bf Y})$ be the corresponding simple $\fS_D$-module.
Indeed $\fS_D$ is identified with the group of monomial matrices of $GL(D)$, and  ${\bf S}({\bf Y})\simeq  L({\bf Y};D)_{{\bf 1^D}}$. It will be  convenient to
denote its class  in $K_0(\fS_n)$ by ${\bf [Y]}$.

By definition a {\it $\fS$-module}  is
a vector space ${\cal P}=\oplus_{n\geq 0}\, {\cal P}(n)$ where 
the component ${\cal P}(n)$ is a finite
dimensional $\fS_n$-module.  
An {\it operad} is a $\fS$-module ${\cal P}$ with some operations,
see \cite{GK} for a precise definition.
Set $K_0(\fS)=\prod_{n\geq 0}\,K_0(\fS_n)$, see \cite{Macbook}. 
The class $[{\cal E}]\in K_0(\fS)$ of a $\fS$-module is
defined by $[{\cal E}]=\sum_{n\geq 0}\,[{\cal E}(n)]$.

For a $\fS$-module ${\cal E}$, the $\fS$-modules $\Res{\cal E}$ and $\Ind{\cal E}$ are  defined by 

\centerline{$\Res\, {\cal E}(n)=\Res_{\fS_n}^{\fS_{n+1}}{\cal E}(n+1)$,}

\centerline{$\Ind\, {\cal E}(n+1)=\Ind_{\fS_n}^{\fS_{n+1}}{\cal E}(n)$}

\noindent for any $n\geq 0$. The functors   $\Res$ and $\Ind$ gives rise  to additive maps
on $K_0(\fS)$ and they are determined by

\centerline{$\Res {\bf [Y]}=\sum_{{\bf Y'}\in \Res\,{\bf Y}}\,[{\bf Y'}]$,}

\centerline{$\Ind {\bf [Y]}=\sum_{{\bf Y'}\in \Ind\,{\bf Y}}\,[{\bf Y'}]$}

\noindent where $\Res\,{\bf Y}$ (respectively
$\Ind\,{\bf Y}$)  is the set of all Young diagrams  obtained by deleting
one box in ${\bf }Y$  (respectively by adding one box to ${\bf Y}$).

A $\fS$-module ${\cal E}$ is called {\it cyclic} if 
${\cal E}=\Res{\cal F}$ for some $\fS$-module ${\cal F}$.

\bigskip
\noindent
{\it 5.3 Schur-Weyl duality}

\noindent
The Shur-Weyl duality is an equivalence of
the categories between  the analytic functors and the  $\fS$-modules. 

For an analytic functor $F$, the corresponding
$\fS$-module  ${\cal F}=\oplus_{n\geq 0}\,{\cal F}(n)$
is defined by

\centerline{ ${\cal F}(n)=F(n)_{\bf 1^n}$.}

\noindent  If $F=\Sigma E$ for some analytic functor $E$, it is clear that 

\centerline{$F(n)_{\bf 1^n}=E(1+n)_{\bf 1^{1+n}}$.}

\noindent Therefore the cyclic analytic functors gives rise to cyclic
$\fS$-modules. Conversely, for any
$\fS$-module ${\cal E}$, the corresponding
analytic functor $Sh_{\cal E}$, which is called a {\it Shur functor},  is defined by:
 
  \centerline{$Sh_{\cal E}(V)=\oplus_{n\geq 0}\,H_0(\fS_n, {\cal E}(n)\otimes V^{\otimes n})$}

\noindent for any $V\in Vect_K$. E.g., 
$Sh_{\bf S(Y)}(D)=L({\bf Y};D)$ for any Young diagram ${\bf Y}$.

 The class of an analytic functor $F$ is 
$[F]=\sum_{n\geq 0}\,[F(n)_{\bf 1^n}]\in K_0({\fS})$.

\begin{lemma}\label{combiYoung}
Let $(F,G,\Theta)$  be a cyclic triple.  Then we have

\centerline{$[\Ker\,\Theta]+[G]= \Ind\circ\Res [\Ker\,\Theta]
=\Ind [F]$}

\end{lemma}
 
 \begin{proof}
 It follows from the fact that
the Schur-Weyl duality establishes the following correspondences:
 
 \begin{tabular}{|l|c|c|}
 \hline
 Categories&  Analytic functors & $\fS$-modules \\
 \hline
  & Cyclic analytic functors & Cyclic $\fS$-modules\\
 \hline
 Functor  &   $\otimes\,\Id$ & $\Ind$\\
 \hline
 Functor & Suspension $\Sigma$ & $\Res$\\

        \hline
 \end{tabular}

\end{proof}

\bigskip

\noindent
{\it 5.4 A list of analytic functors and $\fS$-modules}

We will now provide a list of analytic functors $P$.
For those, the analytic $GL(D)$-module $P(D)$ has been defined,
so the definition of the corresponding functor is easy.
This section is mostly about notations.

For example, for $V\in Vect_K$,  let $T(V)$ be the free 
non-unital associative algebra over the vector space $V$.
The functor $[T,T]$ is the subfunctor defined  by $[T,T,](V)= [T(V),T(V)]$.
Similarly, there are functors $J:V\mapsto J(V)$, $SJ:V\mapsto SJ(V)$
and $CJ:V\mapsto CJ(V)$ which provide, respectively, the free Jordan algebras, the free special Jordan algebras and the free Cohn-Jordan 
algebras. 

Concerning the derivations, we will consider the 
analytic functors
${\cal B}J$, ${\cal B}SJ$, $\Inner SJ$ and $\Inner CJ$.
The last two are functors by Lemma \ref{=[]}.

For the missing spaces, we will consider
the analytic functors of missing tetrads $M=CJ/SJ$,  
of missing derivations $MD=\Inner CJ/\Inner SJ$, which is
a functor by Lemma \ref{=[]}. Also we will consider the
functor of special identities $SI=\Ker J\to SJ$.

Since it is a usual notation, denote by
$Ass$ the associative operad. The other $\fS$-modules will be denoted with calligraphic letters. The Jordan operad is denoted by ${\cal J}$.
As a $\fS$-module, it is defined by 
${\cal J}(D)=J(D)_{\bf 1^D}$. The special Jordan operad
${\cal SJ}$ and the Cohn-Jordan operad ${\cal CJ}$ are defined similarly.
The $\fS$-modules ${\cal M}$,  ${\cal MD}$ and ${\cal SI}$ are the
$\fS$-modules corresponding to the analytic functors $M$, $MD$ and $SI$.

\bigskip
\noindent
{\it 5.5 The cyclic structure on
$T$ and $CJ$}

\noindent
 We will use Lemma \ref{crit} to describe the cyclic structure on the  tensor algebras. It is  more complicated than usual \cite{GK}, because we are looking at a  cyclic structure which is compatible with the free Jordan algebras.
The present approach is connected with \cite{M}.

 The natural map $TV\otimes V\to [TV,TV],\,u\otimes v\mapsto [u,v]$
 for any $V\in Vect_K$ is a natural transformation
 $\Theta_T:T\otimes \Id\to [T,T]$.

\begin{lemma}\label{Tcyclic} The triple
$(T, [T,T], \Theta_T)$ is cyclic.
 \end{lemma}

\begin{proof} Let's begin with a simple observation.
Let $n$ be an integer, let $M=\oplus_{0\leq k\leq n} \,M_k$
be a  vector space and  let $t:M\to M$
be an automorphism of order $n+1$ such that $t(M_k)\subset M_{k+1}$
for any $0\leq k<n$ and $t\,M_n\subset M_0$.
 Then it is clear that the map 

\centerline{$\oplus_{0\leq k<n}\,M_k\to(\,1-t)(M), u\mapsto u-t(u)$}

\noindent is an isomorphism

To prove the lemma, it is enough to prove that the triple
$(T_n, [T,T]_{n+1}, \Theta_T)$ is cyclic for any integer $n$.
Let $V\in Vect_k$ and set $W=k.x_0\oplus V$.
Since we have $[TW,TW]=[TW,W]$,  it follows 

\centerline{$\Sigma\Theta_T( T_n\otimes\Id )(V))=\Sigma [T,T]_{n+1}(V)$.}

Once $T_n W\otimes W$ is identified with $W^{\otimes n+1}$, the
map 

\centerline{$\Theta_T:T_n W\otimes W\to T_{n+1}\,W,\,u\otimes w\mapsto [u,w]$}

\noindent is
identified with the map $1-t$, where
 $t$ is the automorphism of $W^{\otimes n+1}$ defined by
$t(w_0\otimes w_1\otimes w_n)=w_n\otimes w_0\otimes...\otimes w_{n-1}$.
Set 
$M_k=V^{\otimes k}\otimes x_0\otimes V^{\otimes n-k}$ for any $k$.
We have 

\centerline{$\Sigma( T_n\otimes\Id )(V)=\oplus_{0\leq k\leq n} M_k$,
 and $\Sigma\,T_n V\otimes V=\oplus_{0\leq k< n} M_k$.}

\noindent Since $t(M_k)\subset M_{k+1}$
for any $0\leq k<n$ and $t\,M_n\subset M_0$, it follows from the previous observation that $\Theta_T$ induces an isomorphism 
from $\Sigma\,T_n V\otimes V$ to $\Sigma [T,T]_{n+1}(V)$, so the
triple $(T, [T,T], \Theta_T)$ is cyclic.
\end{proof}

Let $V\in Vect_K$. It follows from Lemma \ref{=[]} that $\Inner\,CJV=[CJV,CJV]$. So the natural map 
$CJV\otimes V\to \Inner\,SV,\,u\otimes v\mapsto [u,v]$
is a natural transformation
 $\Theta_{CJ}:CJ\otimes \Id\to \Inner\,CJ$.

\begin{lemma}\label{Scyclic} The triple $(CJ,\Inner\,CJ,\Theta_{CJ})$ is cyclic.  
\end{lemma}

\

\begin{proof} It is clear that the triple $(CJ,\Inner CJ,\Theta_CJ)$ is a direct summand of the previous one, so it is cyclic.
\end{proof}

\noindent
\bigskip{\it 5.6 A preliminary result}

\noindent
According to \cite{Le}, Schreier first proved a  statement
similar to the next Theorem
in the more difficult context of  the free group algebras. Next it has been proved by Kurosh \cite{Ku} and Cohn \cite{Co} in the context of
the free monoid algebras, or, equivalently for the enveloping algebra of a free Lie algebra.

\begin{Kutheorem}
Let $F$ be a free Lie algebra, and let $M$
be a free module. Then any submodule $N\subset M$ is free.
\end{Kutheorem}

Let $D\geq 1$ be an integer and let  $F$ be the free Lie algebra generated by
$F_1:=\fsl_2\otimes K^D$, i.e.
$F$ is a free Lie algebra on $3D$ generators
on which $PSL(2)$ acts by automorphism.
Let ${\cal M}(F,PSL(2))$ be the category of $PSL(2)$-equivariant $F$-modules.
The $F$-action on a module $M\in {\cal M}(F,PSL(2))$ is 
a $PSL(2)$-equivariant map 
$M\otimes F\to M, m\otimes g\mapsto g.m$.  It induces the map

\centerline{$\mu_M:H_0(\fsl_2, M^{ad}\otimes F_1)
\to H_0(\fsl_2, M)$.}

Recall that $X=X^{\fsl_2}\oplus \fsl_2.X$ for any $PSL(2)$-module $X$.

\begin{lemma}\label{lemmamu}
Let $0\to Y\to X\to M\to 0$ be a short exact sequence in
${\cal M}(F,PSL(2))$. Assume that

1. the $F$-module $X$ is  free  and generated by $\fsl_2.X$,

2. $Y$ is  generated by $\fsl_2.Y$.

Then the map the map   $\mu_M$  is an isomorphism.

\end{lemma}

\begin{proof} 
Since $X$ is free, the action $X\otimes F_1\to F.X,
m\otimes g\mapsto g.m$ is an isomorphism, therefore the map  

\centerline{$H_0(\fsl_2, X\otimes F_1) \to H_0(\fsl_2, F.X)$}

\noindent is an isomorphism. Since $F_1$ is of adjoint type,
we have $H_0(\fsl_2, X\otimes F_1)=H_0(\fsl_2, X^{ad}\otimes F_1)$.
As $X$ is generated by $\fsl_2.X$ we have
$H_0(\fsl_2, X/F.X)=0$, so we have $H_0(\fsl_2, F.X)=H_0(\fsl_2, X)$.
It follows that $\mu_X$ is an isomorphism.

By Schreier-Kurosh-Cohn Theorem, $Y$ is also free, and therefore
$\mu_Y$ is also an isomorphism. By the snake lemma, it follows that
$\mu_M$ is an isomorphism.

 \end{proof}

\bigskip Similarly, for $M\in {\cal M}_{\bf T}(\fsl_2\,J(D))$, the action induces a map

\centerline{$\mu_M:H_0(\fsl_2, M^{ad}\otimes (\fsl_2\otimes J_1(D)))
\to H_0(\fsl_2, M)$.}

\begin{lemma}\label{muM}
Let  $M$ be the free $\fsl_2\,J(D)$-module in category 
${\cal M}_{\bf T}(\fsl_2\,J(D))$ generated by one copy of the adjoint module $L(2)$. Then the map $\mu_M$ is an isomorphism.
\end{lemma}

\begin{proof}
Let $F$ be the free Lie algebra of the previous lemma. Any
$PSL(2)$-equivariant isomorphism $\phi: F_1\to \fsl_2\otimes J_1(D)$
gives rise to a Lie algebra morphism $\psi:F\to \fsl_2\,J(D)$, so
$M$ can be viewed as a PSL(2)-equivariant $F$-module.

Let $X\in {\cal M}(F, PLS(2))$ be the free $F$-module generated by
$L(2)$ and let $P$ be the free
$F$-module in category ${\cal M}_{\bf T}(F, PLS(2))$ generated by $L(2)$.
There are natural surjective maps of $F$-modules

\centerline{$X \buildrel \pi \over \longrightarrow P
\buildrel \sigma \over \longrightarrow M$.} 

It is clear that $\Ker\, \pi$ is the $F$-submodule of 
$X$ generated by its $L(4)$-component. Let $K$ be the
$L(4)$-component of $F$. It is clear that  
$\fsl_2\,J(D)=F/R$ where $R$ is the ideal of $F$ generated by $K$.
Therefore $\Ker\,\sigma$ is the $F$-submodule of $P$ generated by 
$K.P$. Since $P$ is in ${\bf T}$, we have $K.P\subset P^{ad}$, therefore
$\Ker\,\sigma$ is generated by its adjoint component.

Set $Y=\Ker\, \sigma\circ\pi$. It follows from the descriptions of
$\Ker\,\pi$ and $\Ker\,\sigma $ that $Y$ is generated by its
$L(2)$ and its $L(4)$ components. Thus the short exact sequence

\centerline{$0\to Y\to X\to M\to 0$}

\noindent satisfies the hypotheses of Lemma \ref{lemmamu}. It follows that $\mu_M$ is an isomorphism.

\end{proof}

\bigskip

\noindent
{\it 5.7 Cyclic structures on $J$ and 
$SJ$}

The natural map
$J(V)\otimes V\to {\cal B}J(V), a\otimes v\mapsto\{a,v\}$,
defined for all $V\in Vect_K$ is indeed a natural transformation
$\Theta_J: J\otimes\Id\to {\cal B}J$.

\begin{lemma}\label{cyclicJ}
The triple $(J,{\cal B}J,\Theta_J)$ is cyclic.
\end{lemma}

\begin{proof} Let $D\geq 0$ be an integer.
Let $M$ be the free $\fsl_2\,J(D)$-module in category 
${\cal M}_{\bf T}(\fsl_2\,J(D))$ generated by one copy $L$ of the 
adjoint module. Let $\fg$ be the Lie algebra
$\fsl_2\,J(D)\ltimes M$. 
Let $\phi$ be a $PSL(2)$-equivariant map 
$\phi:J(D)\otimes K^{1+D} \to \fg$ defined by  the  requirement
that $\phi$ is the identity on $J(D)\otimes K^D$ and
$\phi\vert_{J(D)\otimes x_0}$ is an isomorphism to $L$. 

By Lemma \ref{freeTAG},  $\fsl_2\,J(D)$ is free in the category
${\bf Lie}_{\bf T}$. Therefore $\phi$ extends to a Lie algebra
morphism $\Phi:\fsl_2\,J(D)\rightarrow G$. Note that
$\Phi$ sends $\Sigma J(D)$ to $M$. Since 
$\Sigma J(D)$ a the $\fsl_2\,J(D)$-module generated by
$J(D)\otimes x_0$, it follows that  

  \centerline{$\Sigma \fsl_2\,J(D)\simeq M$ as a $\fsl_2\,J(D)$-module.}

By Lemma \ref{muM}, $\mu_M$ is an isomorphism, which amounts to the fact that  

\centerline{$\Sigma J(J)\otimes K^D\to \Sigma{\cal B}J(D), \ a\otimes v\mapsto\{a,v\}$}

\noindent is an isomorphism. Therefore the triple $(J,{\cal B}(J),\Theta_J)$ is cyclic.

\end{proof}

The natural transformation $\Theta_J$ induces a natural transformation
$\Theta_{SJ}:SJ\otimes\Id\to \Inner\,SJ$. Similarly, we have

\begin{lemma}\label{cyclicSJ}
The triple $(SJ,\Inner\,SJ, \Theta_{SJ})$ is cyclic.

Moreover the natural map
$\Sigma{\cal B}\,SJ(D)\to \Sigma\,\Inner\,SJ(D)$ is an
isomorphism for all $D$. 
\end{lemma}

\begin{proof}
For any $D$, there is a commutative diagram

$\Sigma J(D)\otimes K^D$ \hskip2mm
$\buildrel a \over \twoheadrightarrow $ \hskip2mm
$\Sigma SJ(D)\otimes K^D$   
\hskip5mm $\buildrel b \over \longhookrightarrow$ 
\hskip5mm $\Sigma CJ(D)\otimes K^D$

\hskip1cm$\alpha'\downarrow$ \hskip1.6cm$\alpha$ $\swarrow$  
\hskip8mm  $\searrow\beta$
\hskip3cm$\downarrow\beta'$

$\Sigma{\cal B}SJ(D)$ $\twoheadrightarrow$ $\Sigma{\cal B}SJ(D)$
$\twoheadrightarrow$ $\Sigma \Inner\,SJ(D)$ 
$\hookrightarrow$  $\Sigma \Inner\,CJ(D)$ 

\noindent In the diagram, the horizontal arrows with two heads are obviously
surjective maps, and those with a hook are obviously injective maps.
By Lemma \ref{cyclicJ} the map $\alpha'$ is onto and by 
Lemma \ref{Scyclic} the map $\beta'$ is one-to-one. By diagram chasing,
$\alpha$ and $\beta$ are isomorphisms. Both assertions follow.

\end{proof}

\noindent
{\it 5.8 Cyclicity Theorem}

\noindent
There is a commutative diagram of natural transformations:

\begin{tabular}{c c c c c c c}

$J\otimes\Id$ & $\to$ & $SJ\otimes\Id$ & $\to$ & $S\otimes\Id$ & 
$\to$  & $T\otimes\Id$ \\

$\hskip6mm\downarrow\Theta_J$ && $\hskip6mm\downarrow\Theta_{SJ}$ && 
$\hskip6mm\downarrow\Theta_S$ && $\hskip6mm\downarrow\Theta_T$ \\

${\cal B}J$ & $\to$ & $\Inner SJ$ & $\to$ & $\Inner S$ & 
$\to$  & $[T,T]$ \\

\end{tabular}

\begin{thm}  The four triples $(J,{\cal B}J,\Inner_{SJ}, \Theta_J)$,
$(SJ,\Inner_{SJ}, \Theta_{SJ})$, 

\noindent $(CJ,\Inner_{CJ} \Theta_{CJ})$ and 
$(T,[T,T],\ \Theta_T)$ are cyclic. Moreover the operads
${\cal J}$,  ${\cal SJ}$, ${\cal CJ}$ and ${\cal T}$ are cyclic.

\end{thm}

\begin{proof} The first Assertions follows from Lemmas
\ref{Tcyclic}, \ref{Scyclic}, \ref{cyclicJ} and \ref{cyclicSJ}.
It follows that the $\fS$-modules ${\cal J}$,  ${\cal SJ}$, ${\cal S}$ 
and ${\cal T}$ are cyclic. For an operad, the definition of cyclicity
requires an additional compatibility condition for the action of the cycle, see \cite{GK}. Since this fact will be of  no use here, the proof will be skipped. It is, indeed, formally 
the same as the proof for  the associative operad, see \cite{GK}.

\end{proof}

\bigskip
\noindent
{\it 5.9 Consequences for the free special Jordan algebras}

\noindent
\begin{cor}\label{B=InnerSJ}
We have ${\cal B}SJ(D)=\Inner\,SJ(D)$ for any $D$.
\end{cor}

\begin{proof} Lemma \ref{cyclicSJ} shows that the natural map
$\Sigma {\cal B}SJ =\Sigma \Inner\,SJ$ is an isomorphism. 
Thus the corollary follows from Lemma \ref{SigmaF=0}.

\end{proof}

For any $D$, set ${\cal M}(D)=M(D)_{\bf 1^D}$.

\begin{cor}
The space ${\cal M}(D)$ of multilinear missing tetrads is a $\fS_{D+1}$
-module
\end{cor}

\begin{proof}  By Theorem 1, ${\cal SJ}$ and ${\cal S}$ are
compatibly cyclic. Therefore ${\cal M}(D)$ is a 
$\fS_{D+1}$-module
\end{proof}

For a Young diagram ${\bf Y}$, denote by $c_i({\bf Y})$ the height
of the $i^{th}$ column.

\begin{lemma}\label{c1Young}  Let ${\bf Y}$ be a Young diagram of size 
$D+1$. Assume that 
${\bf S(Y)}$ occurs in the $\fS_{D+1}$-module ${\cal M}(D)$.

1. We have $c_1({\bf Y})\geq 5$ or $c_1({\bf Y})=
c_2({\bf Y})=4$.

2. If moreover  $D= 2$ or $3$ modulo $4$, then
we have $c_1({\bf Y})\leq D-1$.
\end{lemma}

\begin{proof} 
Recall that 

\centerline{${\bf S(Y)}\vert_{\fS_D}=
\oplus_{{\bf Y'}\in \Res {\bf Y}}\,{\bf S(Y')}$.}

\noindent Since $M(3)=0$ by Cohn's reversible Theorem,
$\Res {\bf Y}$ contains no  Young diagram of heigth $<4$.
So it is proved that $c_1(Y)\geq 4$. 
Moreover if $c_1({\bf Y})=4$, removing the bottom  box on the first colum 
does not give rise to a  Young diagram, what forces that
 $c_2({\bf Y})=4$. Assertion 1 is proved.

 Note  that the signature representation of $\fS_D$
 occurs with multiplicity one in ${\cal T}(D)$.
 So if $D=2$ or $D=3$ modulo $4$,
 this representation occurs in the multilinear part of
 $A(D)$, so it does not occur in ${\cal M}(D)$.
 It follows easily that $c_1({\bf Y})\leq D-1$.
 \end{proof}
 
 The Jordan multiplication induces the maps
 $L: CJ_1(D) \otimes M_n(D)\to M_{n+1}(D)$. On the multilinear part,
 it provides a natural map:
 
\centerline{ $L_D:\Ind_{\fS_D}^{\fS_{D+1}}\,{\cal M}(D)\to{\cal M}(D+1)$.}

 \begin{lemma}\label{evenodd} For  $D$ even, the map $L_D$ is onto.
 
 \end{lemma}

\begin{proof} In the course of the proof of
Cohn's Reversible Theorem \cite{McC}, it appears that
$CJ_1(D).CJ_n(D)=CJ_{n+1}(D)$ when $n$ is even. Therefore the
map $L_D$ is onto for $D$ even.
\end{proof}

 \begin{cor}\label{corM}
1. As a $\fS_5$-module, we have ${\cal M}(4)={\bf S(1^5)}$.

2. As a $\fS_6$-module, we have ${\cal M}(5)={\bf S(2,1^4)}$.

3. As a $\fS_7$-module, we have ${\cal M}(6)={\bf S(3,1^4)}^2$

4.  As a $\fS_8$-module, we have 

\noindent\centerline{
${\cal M}(7)= {\bf S(4,1^4)}^{2}\oplus {\bf S(3,2,1^3)}
 \oplus {\bf S(2^2,1^4)}\oplus {\bf S(3,1^5)}$.}

 \end{cor}

 \begin{proof}
 The cases $D=4$ or $D=5$ are easy and the proof for those cases is skipped. We have $\dim{\cal M}(D)=D!/2 -\dim {\cal SJ}(D)$ for any $D\geq 1$.
In \cite{Gl}, it is proved that $\dim {\cal SJ}(6)=330$ and
$\dim {\cal SJ}(7)=2345$. Therefore we have

\centerline{$\dim{\cal M}(6)=30$ and    $\dim{\cal M}(7)=175$.}
 
 Let's consider the case $D=6$.  
 The two Young diagrams of size $7$ and heigth $5$
 are ${\bf Y_1=(3,1^4)}$ and ${\bf Y_2=(2^2,1^3)}$.
 By Lemma \ref{c1Young}, ${\bf S(Y_1)}$ and ${\bf S(Y_2)}$ are
 the only possible simple submodules of the
 $\fS_7$-module ${\cal M}(6)$.
 We have $\dim\,{\bf S(Y_1)}=15$ and $\dim\,{\bf S(Y_2)}=14$.
Since $\dim\,{\cal M}(6)=30$, we have 
 ${\cal M}(6)\simeq {\bf S(3,1^4)}^2$.
 
 For $D=7$ , let's consider the following Young diagrams of size $7$

 ${\bf K_1}=\tiny\yng(3,1,1,1,1)$ 
 \hskip7mm ${\bf K_2}=\tiny\yng(2,1,1,1,1,1)$  \hskip7mm
 ${\bf K_3}= \tiny\yng(2,2,1,1,1)$   
 \hskip7mm${\bf K_4}=\tiny\yng(4,1,1,1)$  
 \hskip7mm${\bf K_5}=\tiny\yng(3,2,1,1)$
 
 \noindent We have 
 $\Ind\,\Res\,{\bf S(3,1^4)}={\bf S(K_1)}^2\oplus {\bf S(K_2)}
\oplus  {\bf S(K_3)}\oplus {\bf S(K_4)}\oplus {\bf S(K_5)}$. 
It follows from Lemma \ref{evenodd} that
 
 \centerline{${\cal M}(7)= \oplus_{1\leq i\leq 5}\,{\bf S(K_i)}^{k_i}$}
 
 \noindent where $k_1\leq 4$ and $k_i\leq 2$ for $2\leq i\leq 5$.
 The list of Young diagrams $Y$  such that  
 $\Res\,Y\subset \{K_1, K_2, K_3, K_4, K_5\}$ is

 ${\bf Y_1}=\tiny\yng(4,1,1,1,1)$ 
 \hskip10mm ${\bf Y_2}=\tiny\yng(3,2,1,1,1)$  \hskip10mm
 ${\bf Y_3}= \tiny\yng(2,2,1,1,1,1)$   
 \hskip10mm${\bf Y_4}=\tiny\yng(3,1,1,1,1,1)$  
 
\noindent  If follows that the $\fS_8$-module ${\cal M}(7)$ can be decomposed as 
 
 \centerline{${\cal M}(7)= \oplus_{1\leq i\leq 4}\,{\bf S(Y_i)}^{m_i}$.}
 
 \noindent 
 Since $\dim{\cal M}(7)=175$ while $\dim {\bf S(Y_1)}=35$, $\dim {\bf S(Y_2)}=64$,
 $\dim {\bf S(Y_3)}=20$, and $\dim {\bf S(Y_4)}=21$, it follows that
 
 \centerline{$35 m_1 + 64 m_2 + 20 m_3 + 21.m_4 =175$.}
 
 \noindent The inequality $k_i\leq 2$ for $i\geq 2$ adds the  constraint $m_i\leq 2$ for any $i$. Thus the only possibility is
 $m_1=2$, $m_2=1$, $m_3=1$, $m_4=1$, and therefore
 
 \centerline{${\cal M}(7)= {\bf S(Y_1)}^{2}\oplus {\bf S(Y_2)}
 \oplus {\bf S(Y_3)}\oplus {\bf S(Y_4)}$}.
 
 \end{proof}

  \begin{cor}\label{corMD}
  We have ${\cal MD}(D)=0$ for $D\leq 4$, and

${\cal MD}(5)={\bf S(2,1^3)}$,

${\cal MD}(6)={\bf S(1^6)}\oplus{\bf S(2,1^4)}\oplus{\bf S(3,1^3)}
\oplus{\bf S(2^2,1^2)}$

${\cal MD}(7)=[{\bf S(2,1^5)} \oplus {\bf S(2^2,1^3)}
\oplus {\bf S(3,1^4)}\oplus {\bf S(3,2,1^2)}\oplus
{\bf S(4,1^3)}]^2$, and
 
 $[{\cal MD}(8)]=
 4\, [{\bf  4,1^4}] + 6 \,[{\bf  3,2,1^3}] +  [{\bf  2^2,4}] +  
 5\, [{\bf  3,1^5}] + 2\,[{\bf 2,1^6}] $
 
\hskip20mm $+ 2\,[{\bf 2,1^6}] + 2\,[{\bf 2^3,1^2}] +
 [{\bf 3^2,1^2}] + 3\,[{\bf 4, 2,1^2}] + 2\,[{\bf 5,1^3}]$,
 
 \noindent where $[{\bf Y}]$ stands for the class of ${\bf S(Y)}$, for 
 any Young diagram ${\bf Y}$.
 
 \end{cor}

\begin{proof} 
The natural transformation $\Theta_{CJ}:CJ\otimes\Id\to \Inner\, CJ$
gives rise to a natual transformation 
$\Theta_M: M\otimes\Id\to MD$.
By Lemmas \ref{cyclicSJ} and \ref{Scyclic}, 
the triple $(M\otimes \Id,MD,\Theta_M)$ is cyclic.
Therefore  the following equality

\centerline{$[{\cal MD}(D+1)]=[\Ind\circ\Res\, {\cal M}(D)]-[{\cal M}(D)]$}

\noindent holds in $K_0(\fS_{D+1})$ by Lemma \ref{combiYoung}.
Since Corollary \ref{corM} provides the character of the 
$\fS_{D+1}$-module ${\cal M(D)}$ for $D\leq 7$, it is possible to
compute the character of ${\cal MD}(D)$ for any $D\leq 8$. The other case
being simpler, some details will be provided for ${\cal MD}(8)$.

 Let's consider the notations of Corollary \ref{corM}. We have
 
 \centerline{$[{\cal M}(7)]=2 \,[{\bf Y_1]+[Y_2]+[Y_3]+[Y_4]}$.}
 
 \noindent It follows that 

\centerline{
 $\Res[{\cal M}(7)]=4\, [{\bf  K_1}] + 2 \,[{\bf  K_2}] + 2\, [{\bf  K_3}] + 
 2\, [{\bf  K_4}]  + [{\bf  K_5}]$, and}
 
 $\Ind\circ\Res[{\cal M}(7)]= 6\, [{\bf  Y_1}] + 7 \,[{\bf  Y_2}] + 2\, [{\bf  Y_3}] +  6\, [{\bf  Y_4}] + 2\,[{\bf 2,1^6}] $
 
\hskip30mm $+ 2\,[{\bf 2,1^6}] + 2\,[{\bf 2^3,1^2}] +$
 $[{\bf 3^2,1^2}] + 3\,[{\bf 4, 2,1^2}] + 2\,[{\bf 5,1^3}]$
 
\noindent from which the formula follows. 

\end{proof}

\bigskip
\noindent
{\it 5.10 Consequence for the free Jordan algebras}

\begin{cor}\label{B=InnerJ}
We have  ${\cal B}_k(J(D))=\Inner_k\,J(D)= \Inner_k\,SJ(D)$ 
 for any $k\leq 8$ and any $D$.
\end{cor}

\begin{proof} By Theorem 1, we have

$\Sigma{\cal B}_k(J(D))\simeq \Sigma J_{k-1}(D)\otimes K^D$, and
$\Sigma{\cal B}_k(SJ(D))\simeq \Sigma SJ_{k-1}(D)\otimes K^D$.

\noindent  By Glennie Theorem,  $J_{k-1}(D)$ and $SJ_{k-1}(D)$ are
isomorphic  for $k\leq 8$. Therefore we have

\centerline{$\Sigma{\cal B}_k(J(D))\simeq \Sigma{\cal B}_k(SJ(D))$}

\noindent for any $k\leq 8$ and any $D$. By Lemma \ref{SigmaF=0}, it follows that ${\cal B}_k(J(D))\simeq {\cal B}_k(SJ(D))$ whenever $k\leq 8$.

Let's consider the commutative diagram

\begin{tabular}{ccc}

\hskip3cm ${\cal B}_k(J(D))$ &  $\buildrel \alpha \over \longrightarrow$ 
&  ${\cal B}_k(SJ(D)$\\

\hskip3cm $\downarrow$ $a$ &  &$\downarrow$ $b$   \\

\hskip3cm $\Inner_k\,J(D)$ & $\buildrel \beta \over \longrightarrow$
& $\Inner_k\,SJ(D)$\\

\end{tabular}

\noindent Observe that all maps are onto. By Corollary \ref{B=InnerSJ}, $b$ is an isomorphism, while it has been proved that $\alpha$ is an isomorphism for
$k\leq 8$. Therefore, the maps $a$ and $\alpha$ are also isomorphism,
what proves Corollary \ref{B=InnerJ}.

\end{proof}

\begin{cor}
The space ${\cal SI}(D)$ of special identities is a $\fS_{D+1}$-module.
\end{cor}

\begin{proof}
By Theorem 1, ${\cal J}$ and ${\cal SJ}$ are
cyclic  and the map ${\cal J}(D)\to{\cal SJ}(D)$ is
$\fS_{D+1}$-equivariant. Therefore ${\cal SI}(D)$ is a 
$\fS_{D+1}$-module.
\end{proof}

For example, let $G$ be the multilinear part of the Glennie Identity. As an
element of ${\cal SI}(8)$, it generates a simple
$\fS_8$ module $M\simeq {\bf S(3^2)}$. What is the
$\fS_9$-module $\hat{M}$ generated by $G$ in ${\cal SI}(8)$? It is clear that there are only two possibilities

A) $M\simeq {\bf S(3^3)}$. In such a case, ${\hat M}=M$.

B) ${\hat M}\simeq  {\bf S(3^2,2,1)}$. 

\noindent If so,  $\Res\,{\hat M}\simeq {\bf S(3^2,2)}\oplus {\bf S(3,2^2,1)}\oplus {\bf S(3^2,1^2)}$. This would provide two  independent new special identities
in $J(4)$. When computing  the simplest of these two identities,
we found a massive expression. Unfortunately, it was impossible to decide if this special identity is zero or not.

\section*{6. The Conjecture 3}

\noindent
Conjecture 2 is quite natural. 
However, the vanishing of $H_*(\fsl_2\,J(D))^{ad}$ 
does not look very tractable. Conjecture 3
is a weaker and better version. 
As a consequence of Theorem 1, it will be proved that
it is nevertheless  enough to deduce Conjecture 1.

\begin{conj}
We have $H_k(\fsl_2\,J(D))^{\fsl_2}=0$  for any $k\geq 1$.
\end{conj}

Note that Conjecture 3 holds for $k=1,2$ and $3$,
as it was proved in section 4.

\begin{thm}
If Conjecture 3 holds for $\fsl_2\,J(1+D)$, then 
Conjecture 1 holds for $\fsl_2\,J(D)$.
\end{thm}

\begin{proof} 
The proof is similar to the proof of Corollary 1.
Assume Conjecture 3 holds for $\fsl_2\,J(1+D)$.

Since $H_*(\fsl_2\,J(D))$ is a summand  in $H_*(\fsl_2\,J(1+D)))$,
it follows that 
  $H_k(\fsl_2\,J(D))^{\fsl_2}=0$ for any $k\geq 1$.
By Lemma \ref{nocenter}, this implies that ${\cal B}J(D)=\Inner J(D)$.
As in the proof  of Corollary 1,  we get that

$({\cal E}_1)$\centerline{$[\lambda [\fsl_2\,J(D)]:L(0)]=1$} 

\noindent where $[\fsl_2\,J(D)]$ denotes the class of 
$\fsl_2\,J(D)$ in ${\cal M}_{an}(GL(D)\times PSL(2))$.

Similarly $\Sigma H_*(\fsl_2\,J(D))$ it is a component of
$H_*(\fsl_2\,J(1+D))$, and therefore
$\Sigma H_*(\fsl_2\,J(D))^{\fsl_2}$ vanishes.
The complex computing $\Sigma H_*(\fsl_2\,J(D))$ is
$\Lambda\, \fsl_2\,J(D)\otimes \Sigma\fsl_2\,J(D)$. 
It follows that

\centerline{
$[(\lambda [\fsl_2\,J(D)]. [\Sigma\,\fsl_2\,J(D)]):L(0)]=0$.}

\noindent Using that $[\Sigma\,\fsl_2\,J(D)]=
[\Sigma{B}J(D)] + [J(D)].[L(2)]$, the previous equation
can be rewritten as:

\centerline{$[\Lambda [\fsl_2\,J(D)]:L(0)][\Sigma{B}J(D)]
+ [\lambda [\fsl_2\,J(D)]:L(2)] [\Sigma J(D)]=0$}

\noindent It has been proved that $[\lambda [\fsl_2\,J(D)]:L(0)]=1$.
Moreover by Theorem 1, we  have $[\Sigma{B}J(D)]=[K^D][\Sigma J(D)]$.
Therefore the previous equation can be simplified into

\centerline{$([K^D]+[\lambda [\fsl_2\,J(D)]:L(2)]).[\Sigma J(D)]=0$.}

\noindent  The ring ${\cal R}_{an}(GL(D))$ is ring of formal series, see
\cite{Macbook} or Section 1.4, and it has no zero divisors.  It follows that

 $({\cal E}_2)$
 \centerline{$[\lambda [\fsl_2\,J(D)]:L(2)]=-[K^D].$}

Using that ${\cal B}J(D)=\Inner J(D)$, Equations ${\cal E}_1$ and
${\cal E}_2$ implies that:

\centerline{
$\lambda([J(D)][L(2)]+[\Inner J(D)]):[L(0)]=1$, and}

\centerline{
$\lambda([J(D)][L(2)]+[\Inner J(D)]):[L(2)]=-[K^D]$.}

\noindent So by Lemma \ref{combi}, Conjecture 3 implies Conjecture 1.

\end{proof}


\begin{thebibliography}{9}



\bibitem{AG} B.~N.~Allison, Y.~Gao, {\it Central quotients and coverings of Steinberg unitary Lie algebras,} Canad. J. Math. 48 (1996) 449-482.

\bibitem{Smirnov} D.~M.~Caveny, O.~M.~Smirnov,  {\it Categories of Jordan Structures and Graded Lie Algebras,} Comm. in Algebra 42 (2014) 186-202.

 \bibitem{Co} P.  M.  Cohn, {\it On a  generalization of the Euclidean algorithm,} Proc. Cambridge Philos. Soc. 57 (1961) 18-30.

\bibitem{GL} H. Garland, J. Lepowski, {\it J. Lie algebra homology and the Macdonald-Kac formulas,} Invent. Math. 34 (1976) 37-76.

\bibitem{GK} E. Getzler,  M. M. Kapranov  {\it Modular operads,} Compositio Math.  110 (1998)  65-125.


\bibitem{Gl}
C.M. Glennie, {\it Some identities valid in special Jordan algebras but not valid in all Jordan algebras,} Pacific J. Math. 16 (1966) 47-59.


\bibitem{HW} G. H. Hardy et E. M. Wright, {\it  An introduction to the theory of numbers},  Clarendon Press, (1938).

\bibitem{HS} G. Hochschild, J.-P. Serre, {\it Cohomology of group extensions,} Transactions of AMS  74 (1953) 110-134.
 
\bibitem{J}
N. Jacobson, {\it Structure and representations of Jordan algebras},   AMS Colloquium Publications 39 (1968).

\bibitem{Kan} I. L. Kantor, {\it Classification of irreducible transitively differential groups,}
Soviet Math. Dokl., 5 (1964) 1404-1407.


\bibitem{Koe}
M. Koecher, {\it The Minnesota Notes on Jordan Algebras and Their Applications},  Springer-Verlag.  Lecture Notes in Mathematics 1710 (1999).

\bibitem{Ko} J.-L, Koszul, {\it Homologie et cohomologie des alg\'ebres de Lie,}  Bulletin de la S. M. F. 78 (1950)  65-127.

\bibitem{Ku}
A. D.  Kurosh, {\it The theory of groups},  Chelsea, New York, 1956.

\bibitem{LM} M. Lau, O. Mathieu, {\it In preparation}.



\bibitem{Le} J. Lewin
{\it Free Modules Over Free Algebras and Free Group Algebras: the Schreier  Technique,}
Transactions of AMS 145  (1969) 455-465.

\bibitem{McC}
K. McCrimmon, {\it A Taste for Jordan algebras,} SpringerVerlag. Universitext 200 (2004).

\bibitem{Mac}
I. G. Macdonald,  {\it Jordan algebras with three generators,} Proc. London Math. Soc. 10 (1960) 
395-408.

\bibitem{Macbook}
I. Macdonald, {\it Symmetric functions and Hall polynomials}, Oxford Math. Monographs.  Oxford University Press, (1995).

\bibitem{M} O. Mathieu {\it Hidden $\Sigma_{n+1}$-actions}, 
Comm. in Math. Phys. 176 (1996) 467-474. 

\bibitem{Med} Yu. A, Medveev. {\it On the nil-elements of a free Jordan algebra,} Sib. Mat. Zh. 26 (2)  (1985) 140–148.

\bibitem{Med2} Yu. A, Medveev {\it Free Jordan algebras,} Algebra and Logic  27 (2) (1988) 110–127.

\bibitem{RNA}  K. A. Zhevlakov, A. M. Slinko, I. P. Shestakov, A. I. Shirshov {\it Rings that nearly Associative,}  Academic Press (1982).

\bibitem{Sverch} S. R. Sverchkov, {\it The Lie algebra of skew-symmetric elements and its application in the theory of Jordan algebras,} Sib. Mat. Zh. 51 (3) (2010) 496-506.

\bibitem{T} J. Tits, {\it Une classe d'alg\'ebres de Lie en relation avec les alg\'ebres de Jordan,}
Indag. Math. 24 (1962) 530-534.

\bibitem{W} E. Weisstein, {\it Necklace}.  From MathWorld--A Wolfram Web Resource. http://mathworld.wolfram.com/Necklace.html 


\bibitem{Z} E. I. Zel'manov, {\it  Absolute
Zero-divisor and algebraic Jordan algebras,} Sib. Mat. Zh.
23 (6) (1982), 841–854.

\bibitem{Z2} E. I. Zel'manov, {\it On prime Jordan algebras. II}, Sib. Mat. Zh., 25 (5)  (1984)  50-61.

\end{thebibliography}
\end{document}